\documentclass[a4paper,fleqn]{cas-sc}

\usepackage[numbers,sort&compress]{natbib}
\usepackage[ruled,vlined,linesnumbered]{algorithm2e}
\usepackage{caption}

\usepackage{amsmath}
\usepackage{amssymb}
\usepackage{amsthm}
\usepackage{amsfonts}
\usepackage{amsbsy}
\usepackage{amstext}
\usepackage{amsfonts}
\usepackage{amstext}
\usepackage{mathrsfs}
\usepackage{graphpap}
\usepackage{graphics,graphicx}
\usepackage[utf8]{inputenc}
\usepackage{subcaption}
\usepackage{multicol}
\usepackage{booktabs}
\usepackage{array}
\usepackage{siunitx}
\usepackage{hyperref}
\usepackage{xcolor}
\usepackage{caption}
\usepackage{float}
\usepackage{subcaption}
\usepackage{xfrac}
\usepackage{ulem}
\setcitestyle{square}

 \usepackage{graphics}
 \usepackage{tikz}
 \usetikzlibrary{external}

\hypersetup{colorlinks = true, allcolors = blue}
\usepackage[nameinlink,noabbrev]{cleveref}

\newtheorem{theor}{Theorem}
\newtheorem{lemma}{Lemma}
\newtheorem{proposition}{Proposition}
\newtheorem{corollary}{Corollary}

\newtheorem{defin}{Definition}
\newtheorem{remark}{Remark}
\usepackage[section]{placeins}
\renewcommand{\thefootnote}{}

\def\tsc#1{\csdef{#1}{\textsc{\lowercase{#1}}\xspace}}
\tsc{WGM}
\tsc{QE}
\tsc{EP}
\tsc{PMS}
\tsc{BEC}
\tsc{DE}

\begin{document}
\let\WriteBookmarks\relax
\def\floatpagepagefraction{1}
\def\textpagefraction{.001}

\shorttitle{Stratified Sampling Algorithms for Machine Learning Methods in Solving Two-scale PDEs}

\shortauthors{Eddel El\'{i} Ojeda Avil\'{e}s, Daniel Olmos Liceaga, Jae-Hun Jung}

\title[mode=title]{Stratified Sampling Algorithms for Machine Learning Methods in Solving Two-scale Partial Differential Equations}





%
\author[1,2]{{Eddel El\'{i} Ojeda Avil\'{e}s}}[style=chinese,
                        orcid=0009-0006-7585-6745
                        ]
\ead{eddelojeda@gmail.com}

\author[1,2]{{Daniel Olmos Liceaga}}[style=chinese,
                        orcid=0000-0002-3536-2162
                        ]
\ead{daniel.olmos@unison.mx}
\author[2]{{Jae-Hun Jung}}[style=chinese,
                        orcid=0000-0001-7923-0674
                        ]

\cormark[1]


\ead{jung153@postech.ac.kr}


\credit{Conceptualization of this study, Methodology, Software}

\affiliation[1]{organization={Universidad de Sonora},
    addressline={Blvd. Luis Encinas y Rosales S/N}, 
    postcode={83000},
    city={Hermosillo}, 
    state={Sonora},
    country={M\'{e}xico}}




\affiliation[2]{organization={Pohang University of Science and Technology},
    city={Pohang},
    postcode={37673}, 
    country={Korea}}

\cortext[cor1]{Corresponding author}




\begin{abstract}
Partial differential equations (PDEs) with multiple scales or those defined over sufficiently large domains arise in various areas of science and engineering and often present problems when approximating the solutions numerically. Machine learning techniques are a relatively recent method for solving PDEs. Despite the increasing number of machine learning strategies developed to approximate PDEs, many remain focused on relatively small domains. When scaling the equations, a large domain is naturally obtained, especially when the solution exhibits multiscale characteristics. This study examines two-scale equations whose solution structures exhibit distinct characteristics: highly localized in some regions and significantly flat in others. These two regions must be adequately addressed over a large domain to approximate the solution more accurately. We focus on the vanishing gradient problem given by the diminishing gradient zone of the activation function over large domains and propose a stratified sampling algorithm to address this problem. We compare the uniform random classical sampling method over the entire domain and the proposed stratified sampling method. The numerical results confirm that the proposed method yields more accurate and consistent solutions than classical methods.
\end{abstract}





\maketitle


\section{Introduction}

\noindent Multiple spatial and temporal scales characterize numerous phenomena in science and engineering. Some examples of such phenomena are presented in climate and ocean sciences \cite{majda2008applied}, materials sciences \cite{fish2010multiscale}, thermal dynamics \cite{yu2002multiscale}, and other fields. The typical mathematical descriptions of such phenomena involve non-linear partial differential equations (PDEs) with multiple scales. This situation has led to the development of the multi-scale method or analysis, comprising various techniques employed to construct approximations of the solutions to problems that depend on different scales simultaneously \cite{hou2003numerical,kevorkian2012multiple}. However, numerically solving problems involving multiple scales is generally problematic. One approach involves introducing fast- and slow-scaling variables  and treating them as independent. While it is challenging, the recent development of strategies for machine learning methods enables the numerical approximation of solutions to such problems by training neural networks \cite{vadyala2022physics, glorot2010understanding, pascanu2013difficulty} that directly  simulate the given PDEs. 

This paper, examines two-scale equations whose solution structures have distinct characteristics: smooth and highly localized in some regions and significantly flat in other regions. The considered PDEs in this paper include the Fisher  equation, a model for the spatial and temporal spread of an advantageous allele in a one-dimensional medium \cite{fisher1937wave}, and the Zeldovich equation, which describes the propagation of a flame in the combustion and detonation of gases \cite{zeldovich1959theorv,gilding2004travelling}. The solutions to these equations are highly localized in small regions and nearly flat in most areas. Due to the localization of the solution, the computational domain could be defined in a small area with numerous sampling points. However, the numerical solution becomes highly sensitive to the boundary values because of the highly nonlinear terms in the equation although the solution changes slowly near the truncated boundaries. Simply imposing exact boundary values at the truncated domain boundaries is insufficient. Thus, defining a sufficiently large computational domain is necessary so that the solution away from the localized area becomes flat enough to minimize the boundary effects.

Consequently, the computational complexity naturally increases for direct numerical simulation. A more serious problem arises when solving problems with neural networks using activation functions, such as the sigmoid function. In the long-range region where the solution is flat with  a very small derivative, the gradients calculated in these areas lead to vanishing gradient problems, resulting in the slow decay of the loss function or no further decay. The two-scale solutions of the Fisher and Zeldovich equations are typically represented using traveling wave solutions. The vanishing gradients of the solution in most areas of the domain of interest result in incorrect wave speeds of the obtained traveling wave solutions, as demonstrated in this paper. Therefore, although the solution is sought in a large domain to avoid truncated boundary effects and the sampling resolution over the entire domain increases, the quality of the solution is not improved due to the intrinsic characteristics of the neural network. Scaling the Fisher equation is a typical solution technique, but the  mentioned problems remain.

This paper aims to propose a novel sampling method that optimizes the weights of a neural network, serving as a numerical approximation for PDEs with two-scale solutions. Various weight-optimization methods for neural networks solve PDEs efficiently \cite{le2015simple, van2022optimally, ioffe2015batch}. However, when PDEs are solved with neural networks over significantly large computational domains, particularly for multi-scale solutions, the optimization becomes challenging. This paper focuses on developing a method to address problems when solving PDEs over relatively large domains. These challenges result from various factors, including the time sensitivity of the solutions and the vanishing gradient problem encountered while optimizing neural network parameters using gradient-based methods for backpropagation \cite{hochreiter1998vanishing, plaut1986experiments}. The effects of increasing the sample size by rescaling the system to larger domains must be examined. Increasing the domain by scaling the equation implies that the samples used during optimization must be larger to guarantee that points in the sample exist outside the regions where the gradient of the function to be optimized is not negligible. This approach is because the optimization process is a gradient-based method applied to a function representing an error measure known as the loss function \cite{wang2020comprehensive}. Only samples with the same density of points per unit of volume were considered to make a fair comparison between the solutions in the numerical examples for an equation before and after scaling. However, this requirement is unnecessary.

To address such problems, we propose a stratified sampling (SS) algorithm to mitigate the vanishing gradient problem by identifying the active and diminishing gradient zones of the activation functions within large domains. The basic idea of the proposed method is to deal with the active and diminishing gradient zones of the loss function differently. The development of sampling techniques such as stratified \cite{liberty2016stratified, may2010data}, systematic \cite{xu2018splitting}, cluster \cite{fraboni2021clustered} and adaptive \cite{hanna2022residual, zeng2022adaptive, wu2023comprehensive, mao2023physics} sampling have emerged as an alternative to uniform random sampling in order to improve the training process of the network. In this paper, uniform random sampling is called classical sampling technique, used as a basis during the analysis of the results obtained by the proposed sampling technique.

This paper consists of the following sections. Section 2 explains the basic methodology of neural networks as a numerical solution to the PDE defined with the initial and boundary conditions \cite{raissi2019physics, cho2021traveling}. Next, Section 3 details the vanishing gradient problem \cite{hochreiter1998vanishing} and its relationship with the diminishing gradient zone \cite{deshpande2018artificial}, which is the primary source of the issues for the problem defined over large domains. This consideration leads to the sampling problem. The second part of this section includes theorems to determine the active and diminishing gradient zones of a trial solution given by a neural network. Then, Section 4 proposes a proper sampling method for training the neural network. The proposed strategy is based on the SS algorithm that first identifies the regions where samples are selected and creates a uniform random sample of points within them \cite{liberty2016stratified}. Further, Section 5 provides numerical examples of three PDEs, called the advection, Fisher, and Zeldovich equations \cite{vadyala2022physics, fisher1937wave, zeldovich1959theorv}, to validate the proposed method and provides comparative studies between the typical uniform and proposed SS methods. Finally, Section 6 concludes with a brief remark on future research. \ref{AppendixA} presents the pseudocode of the proposed sampling method.


\section{Neural Networks for PDEs}
The universal approximation theorem states that neural networks can approximate any continuous function with arbitrary precision if no constraints are placed on the width and depth of the hidden layers \cite{cybenko1989approximation, hornik1989multilayer}. This section explains the basic neural networks for solving PDEs based on the framework of physics-informed neural networks \cite{raissi2019physics}.

In the physics-informed neural network setting, we take the spatiotemporal grid points $(x,t)\in\mathbb{R}^{n}\times\mathbb{R}$ as input for the network and let $U_T(x,t) : \mathbb{R}^{n}\times\mathbb{R} \rightarrow \mathbb{R}$ be a trial solution of the PDE obtained directly as network output. Unless stated otherwise, the spatial variable $x$ is defined in one dimension, although the generalization extends to $n$ dimensions. For the numerical examples, we focus on one-dimensional problems and consider the typical structure of the network \cite{raissi2019physics, cho2021traveling} determined by the following recurrence equations:

\begin{equation} 
\label{NNs}
\left\lbrace \begin{split}
& N_{0}^{(1)}(x,t)=x, \  N_{0}^{(2)}(x,t)=t,\\
& N_{k}^{(j)}(x,t)=\sigma\left( \sum_{i=1}^{h_{k-1}}\omega_{k}^{(j,i)}N_{k-1}^{(i)}(x,t)+b_{k}^{(j)}\right) \ \ \mathrm{for} \ j=1,\ldots ,h_{k} \ \mathrm{and} \ \ k=1,\ldots,F,\\
& N_{F+1}^{(j)}(x,t)=\sum_{i=1}^{h_{F}}\omega_{F+1}^{(j,i)}N_{F}^{(i)}(x,t)+b_{F+1}^{(j)} \ \ \mathrm{for} \ j=1,\ldots ,h_{F+1},\\
& U_{T}(x,t) = \left\lbrace N_{F+1}^{(j)}(x,t) \right\rbrace_{j=1}^{h_{F+1}},
\end{split} \right.
\end{equation}
where $N_{k}^{(j)}(x,t)$ denotes the $j$-th neuron in the $k$-th layer as a function of the input $(x,t)\in\mathbb{R}\times\mathbb{R}$, $h_{k}$ denotes the number of neurons in the $k$-th layer, $\sigma$ represents the activation function, and $\omega_{k}^{(j,i)}$ and $b_{k}^{(j)}$ are the weights and bias that have an impact on the neuron $N_{k}^{(j)}(x,t)$ for each hidden layer $L_{k}$, respectively. 
By definition, the $0$-th layer $N_{0}$ represents the input layer for $x$ and $t$ with two neurons for $x$ $\left(=N_0^{(1)}\right)$ and $t$ $\left(=N_0^{(2)}\right)$, respectively.

Any sigmoid function can be used as an activation function. This paper applies the sigmoid function $\sigma(x)$  given by the logistic function $\sigma(x) = \frac{1}{1+e^{-x}}$. The proposed SS algorithm can be extended to nonsigmoid functions, such as the rectified linear units ReLU$(x)$, LeakyReLU$(x)$, and ReQU$(x)$ functions \cite{ackermann2023deep, lei2022solving}.

The neural network is trained by minimizing the loss function. How the loss function is defined is critical for the performance of the neural network in solving PDEs. In this paper, the loss function used for each $U_{T}$, denoted by
$Loss_{Total}\left(U_{T},S\right)$, is defined as the sum of the residues of the expected values of the PDE and the initial and boundary
conditions over all points in a random sample set $S$ within a domain $\Omega_{x}\times\Omega_{t}\subset\mathbb{R}\times\mathbb{R}$, where $\Omega_{x}$ and $\Omega_{t}$ denote the spatial and temporal domains, respectively \cite{raissi2019physics, wang2020comprehensive}. We consider the partition of $S$ by the initial ($S_{IC}$), boundary ($S_{BC}$), and interior ($S_{PDE}$) domains. 
\begin{equation}
\label{partS}
S_{IC} = S\cap \left(\Omega_{x}\times\partial\Omega_{t}\right), \ S_{BC} = S\cap \left(\partial\Omega_{x}\times\Omega_{t}\right), \  S_{PDE} = S\cap \left(\Omega_{x}^{\mathrm{\circ}}\times\Omega_{t}^{\mathrm{\circ}}\right),
\end{equation}
where $\Omega_{x}^{\mathrm{\circ}}$ and $\Omega_{t}^{\mathrm{\circ}}$ denote the interior spatial and temporal domains, respectively.

We let $Loss_{IC}$, $Loss_{BC}$ and $Loss_{PDE}$ be the loss functions corresponding to $S_{IC}$, $S_{BC}$ and $S_{PDE}$, respectively, 
with the mean squared error (MSE) between the expected value of $U(x,t)$ and $U_T(x,t)$  on each domain. We take the total loss $Loss_{Total}$ as the sum of all losses in subdomains as follows: 
\begin{align*}
\label{Loss_short}
Loss_{Total} = Loss_{IC}+Loss_{BC}+Loss_{PDE}.
\end{align*}

This work considers PDEs that yield {\color{black}highly localized} solutions defined in a small region within a larger  domain. For such solutions the boundary conditions are critical. The numerical solutions near the truncated domain boundaries are highly sensitive to the boundary values imposed when the given PDEs contain highly nonlinear terms. For the one-dimensional problem, $\Omega_{x}\times\Omega_{t}=\left[ x_{L},x_{R}\right]\times\left[ t_{0},t_{f}\right]\subset\mathbb{R}\times\mathbb{R}$, proper values must be imposed at both boundaries, left and right, although they are not necessarily physical boundary conditions. 
For this case, we further split the loss function into $S_{LBC}=\left\lbrace x_{L}\right\rbrace\times\Omega_{t}$ and $S_{RBC}=\left\lbrace x_{R}\right\rbrace\times\Omega_{t}$. Then, the total loss, $Loss_{Total}$, is given by
\begin{equation}
\label{Loss}
Loss_{Total} = Loss_{IC}+Loss_{LBC}+Loss_{RBC}+Loss_{PDE}.
\end{equation}

Figure \ref{fig1} depicts a schematic representation of the overall architecture of the model used in the optimization process of the loss function $Loss_{Total}$ performed during network training. Given the nature of the equations in this paper, we only consider $U_{T}:\mathbb{R}^{2}\rightarrow\mathbb{R}$ (i.e., $h_{F+1}=1$). In this representation, $(x,t)$ is input first in the network, and the corresponding operations are performed on each path through the network to obtain the trial solution $U_{T}(x,t)$ \cite{lu2019deeponet}. The loss functions for the initial condition, boundary conditions, and differential equation are calculated from $U_{T}(x,t)$. Finally, the total loss function is obtained by taking the sum of the previously calculated loss functions. \\ \newline

\begin{figure}[h]
	\centering
        \includegraphics[scale=.8]{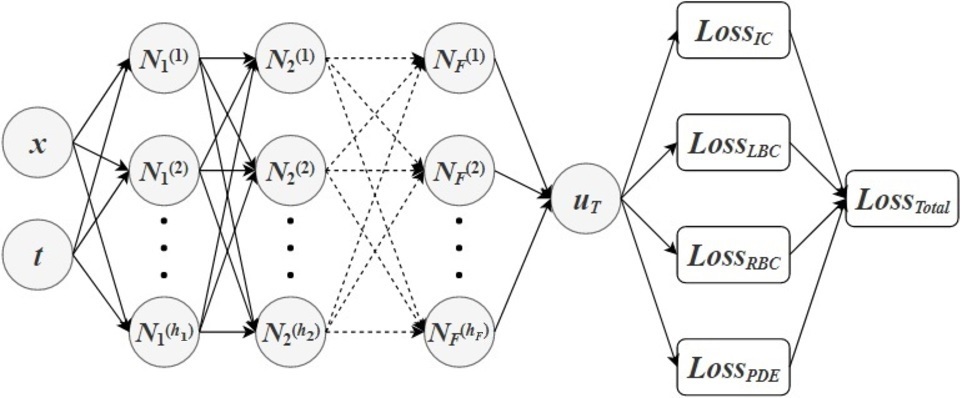}
	\caption{Schematic representation of the overall architecture of the model, considering the trial solution $U_{T}$ as a function of the input $(x,t)$ obtained from the neural network of $F$ hidden layers.}
	\label{fig1}
\end{figure}

For the next part in the optimization process, known as back-propagation, it is necessary to compute all partial derivatives of the loss function $Loss_{Total}$, defined by Eq. (\ref{Loss}), with respect to the weights $\omega_{k}^{(j_{k},i_{k-1})}$ and biases $b_{k}^{(j_{k})}$. In this paper we use automatic differentiation for all such partial derivatives and apply the Adam optimizer to update  $\omega_{k}^{(j_{k},i_{k-1})}$ and $b_{k}^{(j_{k})}$: 
\begin{equation}
\label{optimadampde}
\left\lbrace \omega_{k}^{(j_{k},i_{k-1})},b_{k}^{(j_{k})}\right\rbrace_{k=1}^{F+1} \leftarrow \mathrm{AdamOptimizer} \left( \left\lbrace \omega_{k}^{(j_{k},i_{k-1})},b_{k}^{(j_{k})},\frac{\partial Loss_{Total}}{\partial \omega_{k}^{(j_{k},i_{k-1})}},\frac{\partial Loss_{Total}}{b_{k}^{(j_{k})}} \right\rbrace_{k=1}^{F+1} ,\eta\right),
\end{equation}
where $\eta$ is the learning rate used as the coefficient for the gradient of $Loss_{Total}$ with respect to $\left\lbrace \omega_{k}^{(j_{k},i_{k-1})},b_{k}^{(j_{k})}\right\rbrace_{k=1}^{F+1}$. 

For the experiments, we applied the value of $\eta\leq 1$. Moreover, to reduce the time required to train the network, we took the value of $\eta$ as a function of the loss by allowing larger (e.g., $10^{-3}$ for the Adam optimizer) or smaller ($10^{-4}$) $\eta$ values while modifying the network parameters, depending on the distance to their optimal values that minimize $Loss_{Total}$. The distance to the optimal value is not known in general; hence, the absolute gradient of the loss function is employed as an indicator of the distance. Although it is possible to code the algorithm so that it automatically chooses the best $\eta$ value, due to the expected computational cost, we chose the $\eta$ value from the set of predefined values (e.g., between $10^{-4}$ and $10^{-3}$), depending on the condition based on the gradient magnitude of $Loss_{Total}$. Additionally, for initializing the network weights, we use the same optimizer with the error due to the initial condition as shown below:
\begin{equation}
\label{optimadamic}
\left\lbrace \omega_{k}^{(j_{k},i_{k-1})},b_{k}^{(j_{k})}\right\rbrace_{k=1}^{F+1} \leftarrow \mathrm{AdamOptimizer} \left( \left\lbrace \omega_{k}^{(j_{k},i_{k-1})},b_{k}^{(j_{k})},\frac{\partial Loss_{IC}}{\partial \omega_{k}^{(j_{k},i_{k-1})}},\frac{\partial Loss_{IC}}{b_{k}^{(j_{k})}}\right\rbrace_{k=1}^{F+1} ,\eta\right). 
\end{equation}


\section{Vanishing Gradients}

One of the primary challenges of using machine learning methods to approximate numerical solutions of differential equations is the vanishing gradient problem \cite{hochreiter1998vanishing}, which occurs while training artificial neural networks with gradient-based algorithms and backpropagation. This problem is crucial when solving the two-scale PDEs considered in this work, such that highly sharp solutions are defined in a small region within a larger domain. The solution to this problem is highly localized, whereas it is flat in most other regions in the domain. 

In each epoch, the neural network parameters are updated for the partial derivatives of the loss function. Thus, if some of the derivatives in the gradient of the loss function are sufficiently small, one or several network parameters may cease to update their values, potentially leading to the worst-case scenario that the neural network is hindered from further training. This situation becomes increasingly likely when using backpropagation to compute the gradient of the loss function \cite{plaut1986experiments,montavon2019layer}. As the algorithm relies on applying the chain rule layer by layer, it inherently multiplies the values of the partial derivatives acquired in each layer to calculate the gradient of the loss function. Therefore, the resulting partial derivative becomes exponentially small as it passes through multiple layers, where its derivative is calculated to be quite small. This behavior that affects the first network layer to a greater extent is known as the vanishing gradient problem \cite{pascanu2013difficulty, hochreiter1998vanishing}. A higher number of network layers results in a greater effect on the weights of the first few layers \cite{hanin2018neural}. This problem emerges when using a sigmoid function as an activation function due to the characteristics of its derivatives. Specifically, this occurs with the logistic function employed in this work. As stated, this problem is not exclusive to the activation function in this study. Thus, substituting the activation function with another sigmoid function does not solve this problem.

Several approaches have been developed to address this problem in updating the network parameters by increasing the loss function value. In addition to employing an activation function without this problem \cite{kong2017hexpo}, some proposed techniques involve modifying the loss function by assigning a greater or lesser weight to the error from the trial solution value concerning the initial and boundary conditions or differential equations. Different factors can be considered in these techniques, such as the current parameter values, network structure, selected samples \cite{van2022optimally}, initial weights, and preprocessed data \cite{glorot2010understanding, le2015simple,squartini2003preprocessing}. Other alternatives are to determine which points of the selected samples generate a high value when evaluated at the loss gradient and ignore sample points that do not meet this condition \cite{mo2021quantifying} or use batch normalization through layer inputs when recentering and rescaling \cite{ioffe2015batch}. Although these methods have demonstrated promising results, none can entirely prevent this problem. Therefore, it remains relevant to create alternative techniques to prevent the vanishing gradient problem or minimize the effect as much as possible. This work aims to mitigate this problem and improve the training process for solutions based on the SS techniques.

The proposed SS algorithms aim to identify the domain regions where the gradient of the loss function is large enough in the first hidden layer. We focus on this layer because it is the one most affected by the vanishing gradient problem, causing its values to be more prone to remain unchanged throughout training. In addition, due to the network structure, its output can be interpreted as a function of traveling waves of constant speed given by the neurons in the first layer. Therefore, improving the training of this layer allows us to identify sharp solutions better. In that sense, we propose algorithms guaranteeing that the parameters most affected by the vanishing gradient problem $\left(\omega_{1}^{(j,i)} \ \mathrm{and} \ b_{1}^{(j)}\right)$ remain updated during training. Moreover, the remaining parameters are expected to optimize their values similarly to the way they normally would.

To better understand this idea, it is necessary to study the role of each neuron in the optimization process. The proposed optimization process consists of two stages: (i) the network parameter update stage, considering only the errors from the initial condition, and (ii) the optimized weight stage to re-optimize weights, simultaneously considering errors from the initial and boundary conditions and the given PDE. The first stage approximates a univariate function with a composition of sigmoid functions. Thus, we must analyze the characteristics presented in a trial solution for cases where it is desired to approximate this type of function before studying the general case of the second stage of optimization.

\subsection{Preliminaries}
The necessary preliminaries are provided for the analysis. Some definitions and propositions are based on information from previous work \cite{rennie1969stirling,broder1984r,minai1993derivatives}.

\begin{defin}
\label{Defsn}
The number of ways to partition a set of $m$ distinguishable objects into $n$ non-empty subsets, with $n\leq m$, is called the Stirling number of the second kind (or Stirling partition number), and is denoted by $S(m ,n)$. In particular, the value of $S(m,n)$ can be calculated as
\begin{equation}
\label{S_mn}
S(m,n)=\frac{1}{n!}\sum_{k=0}^{n}(-1)^{k}{n \choose k}(n-k)^{m}.
\end{equation}
\end{defin}

Note that from Definition \ref{Defsn} it follows that $S(m,n)\geq 0$ for any $m\ge 0$ and $n\ge 0$ with $n \le m$. Also it follows that $S(m,1)=S(m,m)=1$ for any value of $m\in\mathbb{N}$. By definition, we also have  $S(m,n) = 0$ when $n = 0$ and $m\ge 1$ and $S(m,n) = 1$ when $n = m = 0$. \\ 

\begin{proposition}
\label{prop1}
Given $m,n\in\mathbb{N}$, with $m\leq n$, the Stirling numbers of the second kind obey the following recurrence relation:
\begin{equation}
\label{propsmn}
S(m+1,n)=nS(m,n)+S(m,n-1).
\end{equation}
\end{proposition}
\begin{proof}
Let $m,n\in\mathbb{N}$, with $m\leq n$, be arbitrary. From Eq. (\ref{S_mn}) it follows that
\begin{align*}
S(m+1,n)&=\frac{1}{n!}\sum_{k=0}^{n}(-1)^{k}{n \choose k}(n-k)^{m+1}\\
&=\frac{n}{n!}\sum_{k=0}^{n}(-1)^{k}{n \choose k}(n-k)^{m}+\frac{1}{n!}\sum_{k=0}^{n}(-1)^{k+1}k{n \choose k}(n-k)^{m}\\
&=nS(m,n)+\frac{1}{n!}\sum_{k=1}^{n}(-1)^{k-1}\frac{n!}{(k-1)!(n-k)!}(n-k)^{m}\\
&=nS(m,n)+\frac{1}{(n-1)!}\sum_{k=1}^{n}(-1)^{k-1}\frac{(n-1)!}{(k-1)![n-1-(k-1)]!}[n-1-(k-1)]^{m}\\
&=nS(m,n)+\frac{1}{(n-1)!}\sum_{k=1}^{n}(-1)^{k-1}{n-1 \choose k-1}[n-1-(k-1)]^{m}\\
&=nS(m,n)+\frac{1}{(n-1)!}\sum_{k=0}^{n-1}(-1)^{k}{n-1 \choose k}(n-1-k)^{m}\\
&=nS(m,n)+S(m,n-1).
\end{align*}
\end{proof}

\begin{corollary}
\label{cor1}
Given $k,m,n\in\mathbb{N}$, with $k\leq m\leq n$, it follows that $S(m,k)\leq S(n,k)$.
\end{corollary}
\begin{proof}
Let $k,m,n\in\mathbb{N}$, with $k\leq m\leq n$, be arbitrary. If $m=n$, clearly $S(m,k)=S(n,k)$. Now consider that  $m<n$, so in particular $k<n$. Then Proposition \ref{prop1} implies that
\begin{equation*}
S(n,k) = kS(n-1,k)+S(n-1,k-1) \geq S(n-1,k).
\end{equation*}
Thus, by applying the previous inequality iteratively $n-m$ times, we obtain
\begin{equation*}
S(n,k) \geq S(n-1,k) \geq \ldots \geq S(m+1,k) \geq S(m,k).
\end{equation*}
\end{proof}
\begin{defin}
\label{DefS}
A sigmoid function is a bounded, differentiable, real function that is defined for all real input values and has a non-negative derivative at each point and exactly one inflection point.
\end{defin}
\noindent

The above definition implies that the graph of the sigmoid function is an S-shaped curve. It follows from Definition \ref{DefS} that the logistic function $\sigma (x)=\frac{1}{1+e^{-x}}$ is an example of the sigmoid function. In this work, the logistic function is used for the network activation function. 

Using these definitions and propositions, we have the following formula for the $n$-th derivative of the sigmoid function.
 
\begin{proposition}
\label{proprep}
The $n$-th derivative of the sigmoid function $\sigma (x)=\frac{1}{1+e^{-x}}$ is given by
\begin{equation}
\label{Sigma_n}
\sigma^{(n)}(x)=\sum_{k=1}^{n+1}(-1)^{k+1}(k-1)!S(n+1,k)\sigma^{k}(x).
\end{equation}
\end{proposition}
\begin{proof}
We prove the above by mathematical induction. In particular, since $S(1,1)=S(2,1)=S(2,2)=1$, and $\sigma^{(0)}(x)=\sigma(x)$, it is clear that
\begin{align*}
\sigma^{(0)}(x)&=\sum_{k=1}^{1}(-1)^{k+1}(k-1)!S(1,k)\sigma^{k}(x)=(-1)^{2}(0!)S(1,1)\sigma(x)=\sigma(x),
\end{align*}
and
\begin{align*}
\sigma^{(1)}(x)&=\sum_{k=1}^{2}(-1)^{k+1}(k-1)!S(2,k)\sigma^{k}(x)\\
&=(-1)^{2}(0!)S(2,1)\sigma(x)+(-1)^{3}(1!)S(2,2)\sigma^{2}(x)\\
&=\sigma(x)-\sigma^{2}(x)\\
&=\sigma^{\prime}(x).
\end{align*}
\noindent Thus, Eq. (\ref{Sigma_n}) is true for $n=1$ and $n=2$. Now, assume that Eq. (\ref{Sigma_n}) is true for some $n$. From Proposition \ref{prop1} and the fact that $S(n+1,1)=S(n+2,1)=1$ and $S(n+1,n+1)=S(n+2,n+2)=1$ it follows that
\begin{align*}
\sigma^{(n+1)}(x)&=\frac{d}{dx}\sigma^{(n)}(x)\\
&=\frac{d}{dx}\sum_{k=1}^{n+1}(-1)^{k+1}(k-1)!S(n+1,k)\sigma^{k}(x)\\
&=\sum_{k=1}^{n+1}(-1)^{k+1}(k-1)!S(n+1,k)k\sigma^{k-1}(x)\left[\sigma(x)-\sigma^{2}(x)\right]\\
&=\sum_{k=1}^{n+1}(-1)^{k+1}(k!)S(n+1,k)\sigma^{k}(x)+\sum_{k=1}^{n+1}(-1)^{k+2}(k!)S(n+1,k)\sigma^{k+1}(x)\\
&=(-1)^{2}(1!)S(n+1,1)\sigma(x)+\sum_{k=2}^{n+1}(-1)^{k+1}(k!)S(n+1,k)\sigma^{k}(x)\\
& \ \ \ \ +\sum_{k=1}^{n}(-1)^{k+2}(k!)S(n+1,k)\sigma^{k+1}(x) +(-1)^{n+3}(n+1)!S(n+1,n+1)\sigma^{n+2}(x)\\
&=(-1)^{2}(1!)S(n+2,1)\sigma(x)+\sum_{k=2}^{n+1}(-1)^{k+1}(k!)S(n+1,k)\sigma^{k}(x)\\
& \ \ \ \ +\sum_{k=2}^{n+1}(-1)^{k+1}(k-1)!S(n+1,k-1)\sigma^{k}(x) +(-1)^{n+3}(n+1)!S(n+2,n+2)\sigma^{n+2}(x)\\
&=(-1)^{2}(1!)S(n+2,1)\sigma(x)+\sum_{k=2}^{n+1}(-1)^{k+1}(k-1)!\left[kS(n+1,k)+S(n+1,k-1)\right]\sigma^{k}(x) \\
& \ \ \ \ +(-1)^{n+3}(n+1)!S(n+2,n+2)\sigma^{n+2}(x)\\
&=(-1)^{2}(1!)S(n+2,1)\sigma(x)+\sum_{k=2}^{n+1}(-1)^{k+1}(k-1)!S(n+2,k)\sigma^{k}(x) \\
& \ \ \ \ +(-1)^{n+3}(n+1)!S(n+2,n+2)\sigma^{n+2}(x)\\
& = \sum_{k=1}^{n+2}(-1)^{k+1}(k-1)!S(n+2,k)\sigma^{k}(x).
\end{align*}
Therefore, for all $n\in\mathbb{N}$ it is true that $\sigma^{(n)}(x)=\sum_{k=1}^{n+1}(-1)^{k+1}(k-1)!S(n+1,k)\sigma^{k}(x)$.
\end{proof}

Using the above propositions, we can prove the following lemma. 
\begin{lemma}
\label{Propic}
Given $\epsilon \ll 1$, $n\in\mathbb{N}$ and $\sigma (x)=\frac{1}{1+e^{-x}}$, there exists  $\delta_{\epsilon_{n}}>0$ such that, if $\vert x\vert >\delta_{\epsilon_{n}}$, then $\vert \sigma^{(m)}(x)\vert <\epsilon$ for all $m\in\left\lbrace 1,\ldots ,n\right\rbrace$.
\end{lemma}
\begin{proof}
Let $\epsilon \ll 1$ and $m,n\in\mathbb{N}$, $m\leq n$, be arbitrary. By letting $\epsilon_n$ and $\delta_{\epsilon_n}$ be as the followings
\begin{equation}
\label{ELemnma}
\epsilon_{n}=\frac{\epsilon}{(n+1)!\max_{k\in \lbrace 1,\ldots n+1\rbrace} \lbrace S(n+1,k)\rbrace},
\end{equation}
and 
\begin{equation}
\label{deltaELemnma}
\delta_{\epsilon_{n}}=\ln{\left( \frac{1-\epsilon_{n}}{\epsilon_{n}}\right)},
\end{equation}
Corollary \ref{cor1} and Proposition \ref{proprep} imply that for any $ x< -\delta_{\epsilon_{n}}$ it holds that 

\begin{align*}
\left\vert \sigma^{(m)}(x) \right\vert &=\left\vert \sum_{k=1}^{m+1}(-1)^{k+1}(k-1)!S(m+1,k)\sigma^{k}(x)\right\vert \\
& \leq \sum_{k=1}^{m+1} \left\vert (k-1)!S(m+1,k)\sigma^{k}(x)\right\vert \\
& \leq \sum_{k=1}^{m+1} \max_{k\in \lbrace 1,\ldots m+1\rbrace} \left\lbrace (k-1)!S(m+1,k)\sigma^{k}(x)\right\rbrace \\
& \leq \sum_{k=1}^{m+1} m! \max_{k\in \lbrace 1,\ldots m+1\rbrace} \lbrace S(m+1,k)\rbrace  \sigma\left(x\right)  \\
& < (m+1)! \max_{k\in \lbrace 1,\ldots m+1\rbrace} \lbrace S(m+1,k)\rbrace \sigma\left(-\delta_{\epsilon}\right) \\
& \leq (n+1)! \max_{k\in \lbrace 1,\ldots n+1\rbrace} \lbrace S(n+1,k)\rbrace \frac{1}{1+e^{\ln{\left( \frac{1-\epsilon_{n}}{\epsilon_{n}}\right)}}}\\
& = (n+1)! \max_{k\in \lbrace 1,\ldots n+1\rbrace} \lbrace S(n+1,k)\rbrace \epsilon_{n}\\
& = \epsilon .
\end{align*}
Since, $\sigma (x)= \frac{1}{1+e^{-x}}=\frac{e^{x}}{e^{x}+1}=1-\frac{1}{1+e^{x}}=1-\sigma (-x)$, it is clear that $\sigma^{(n)}(x)=(-1)^{n+1}\sigma^{(n)}(-x)$, thus, for any $ x> \delta_{\epsilon_{n}}$ we have that
\begin{align*}
\left\vert \sigma^{(m)}(x) \right\vert &= \left\vert (-1)^{m+1}\sigma^{(m)}(-x) \right\vert = \left\vert \sigma^{(m)}(-x) \right\vert \leq \max_{x< -\delta_{\epsilon_{n}}}\left\vert\sigma^{(m)}\left(x\right)\right\vert < \epsilon.
\end{align*}
Therefore, for all $\vert x\vert >\delta_{\epsilon_{n}}$ it holds that $\vert \sigma^{(m)}(x)\vert <\epsilon$ for all $m\in\lbrace 1,\ldots ,n\rbrace$.
\end{proof}

\subsection{Active Gradient Zone Analysis}

To analyze the proposed method, we consider the interval where all derivatives of the sigmoid function are nonnegligible (i.e., the active gradient zone), and the complement of this interval (i.e., the diminishing gradient zone). This definition holds for any function, not just $\sigma(x)$. These definitions depend on which values of $\sigma^{(n)}(x)$ are considered negligible for all $n\in \mathbb{N}$, and the active and diminishing gradient zones can generally be defined when a reference value, such as $\epsilon$, is given. Lemma \ref{Propic} provides the condition of choosing $\delta_{\epsilon_n}$ associated with the active and diminishing gradient zones of $\sigma(x)$ for the given value of $\epsilon \ll 1$, considering not all, but only the first $n$ derivatives of $\sigma(x)$, for a given $n\in \mathbb{N}$, in Eqs. $(\ref{ELemnma})$ and $(\ref{deltaELemnma})$. This paper considers using only the first $n$ derivatives to determine these regions. Figure \ref{fig2} presents the active and diminishing gradient zones of $\sigma (x)$ for its first four derivatives and $\epsilon\approx 10^{-3}$.
\begin{figure}[h]
	\centering
	\includegraphics[scale=.65]{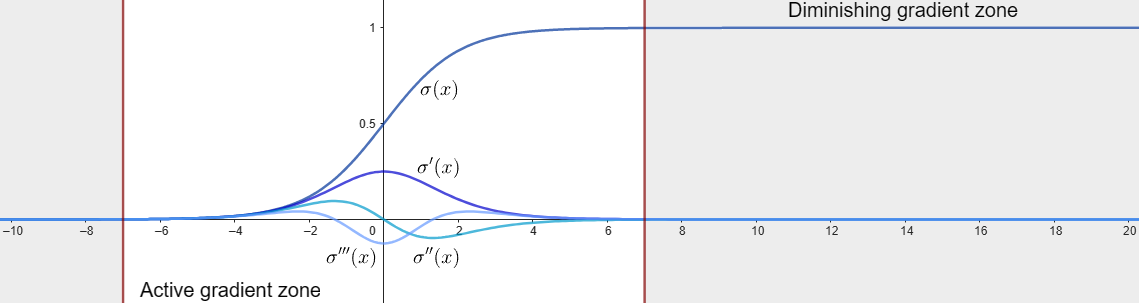}
	\caption{Active and diminishing gradient zones of the sigmoid function $\sigma (x)=\frac{1}{1+e^{-x}}$.}
	\label{fig2}
\end{figure}

Any trial solution $U_{T}(x,t)$ of a PDE of order $n\in\mathbb{N}$ such as those analyzed in this work, given by Eq. $(\ref{NNs})$, has as its activation function $\sigma (x)=\frac{1}{1+e^{-x}}$. Thus, between the equation, the initial and boundary conditions, all the partial derivatives of $U_{T}(x,t)$ that appears in the calculation of $Loss_{Total}$ are $\tfrac{\partial^{m} U_{T}(x,t)}{\partial x^{p}\partial t^{m-p}}$ for $0\leq p\leq m\leq n$. Therefore, all the partial derivatives of $U_{T}(x,t)$ with respect to $\omega_{1}^{(j,i)}$ and $b_{1}^{(j)}$ that appears when applying the Adam optimizer to $Loss_{Total}$ are $\tfrac{\partial^{m+1} U_{T}(x,t)}{\partial\omega_{1}^{(j,i)}\partial x^{p}\partial t^{m-p}}$ and $\tfrac{\partial^{m+1} U_{T}(x,t)}{\partial b_{1}^{(j)}\partial x^{p}\partial t^{m-p}}$ for $0\leq p\leq m\leq n$, respectively. Since our goal is to improve the optimization process $\omega_{1}^{(j,i)}$ and $b_{1}^{(j)}$ through stratified sampling, it is necessary to first analyze previous partial derivatives of $U_{T}(x,t)$, considering the active and diminishing gradient zones of $\sigma\left(\omega_{1}^{(j,1)}x+\omega_{1}^{(j,2)}t+ b_{1}^{(j)}\right)$, for a given $\epsilon\ll 1$ and $j=1,\ldots ,h_{1}$. 

\begin{theor}
\label{Coric}
Let $U_{T}(x,t)$ be given by Eq. $(\ref{NNs})$, $\epsilon \ll 1$ and $n\in\mathbb{N}$. For any $j\leq h_{1}$, there exists $\delta_{\epsilon_{j}}>0$ such that for all
\begin{equation}
\label{dgztheo1}
(x,t)\notin \left\lbrace (x,t)\in\Omega_{x}\times\Omega_{t} : \left\vert \omega_{1}^{(j,1)}x+\omega_{1}^{(j,2)}t+ b_{1}^{(j)}\right\vert < \delta_{\epsilon_{j}}\right\rbrace ,
\end{equation}
it holds that 
\begin{equation}
\label{Cor8}
\left\vert\frac{\partial^{m+1} U_{T}(x,t)}{\partial\omega_{1}^{(j,i)}\partial x^{p}\partial t^{m-p}}\right\vert <\epsilon \ \ \ \mathrm{and} \ \ \ \left\vert\frac{\partial^{m+1} U_{T}(x,t)}{\partial b_{1}^{(j)}\partial x^{p}\partial t^{m-p}}\right\vert <\epsilon ,
\end{equation}
for $i=1,2,$ and $0\leq p\leq m\leq n$.
\end{theor}

\begin{proof}

Let $U_{T}(x,t)$ be given by Eq. $(\ref{NNs})$, $\epsilon \ll 1$ and $j\leq h_{1}$. Consider $\delta_{is}$ as the Kronecker delta, $P(A,q)$ as the set of all the partitions of the set $A$ into $q$ non-empty sets $\left(\lambda_1,\ldots,\lambda_{q}\right)$. From the multivariate Fa{\`a} di Bruno's formula \cite{faa1855sullo, di1857note, constantine1996multivariate} the partial derivatives in Eq. $(\ref{Cor8})$ are determined by the following recurrence equations

\begin{equation*}
\left\lbrace \begin{split}
& \frac{\partial^{\left\vert \lambda_{l}^{2}\right\vert} N_{1}^{\left(j_{1}\right)}(x,t)}{\prod_{s\in\lambda_{l}^{2}}\partial v_{s}}  =   \sigma^{\left(\left\vert \lambda_{l}^{2}\right\vert\right)}\left( \omega_{1}^{(j,1)}x+\omega_{1}^{(j,2)}t+b_{1}^{(j)}\right) \prod_{s\in\lambda_{l}^{2}} \left(\delta_{(m+1)s} x +\sum_{i=p+1}^{m} \delta_{is} \omega_{1}^{(j,2)} + \sum_{i=1}^{p} \delta_{is}\omega_{1}^{(j,1)} \right)  , \\
&  \frac{ \partial^{\left\vert \lambda_{l}^{k+1}\right\vert} N_{k}^{\left(j_{k}\right)}(x,t)}{\prod_{s\in\lambda_{l}^{k+1}}\partial v_{s}}  =  \sum_{q=1}^{\left\vert \lambda_{l}^{k+1}\right\vert} \sum_{\lambda^{k}\in P\left( \lambda_{l}^{k+1},q\right)} \sigma^{(q)}\left( \sum_{i=1}^{h_{k-1}}\omega_{k}^{\left(j_{k},i\right)}N_{k-1}^{(i)}(x,t)+b_{k}^{\left(j_{k}\right)}\right) \prod_{r=1}^{q} \frac{ \sum_{i=1}^{h_{k-1}} \omega_{k}^{\left( j_{k},i\right)} \partial^{\left\vert \lambda_{r}^{k}\right\vert} N_{k-1}^{(i)}(x,t)}{\prod_{s\in\lambda_{r}^{k}}\partial v_{s}}  , \\
&  \frac{\partial^{m+1} N_{F}^{\left(j_{F}\right)}(x,t)}{\partial\omega_{1}^{(j,1)}\partial x^{p}\partial t^{m-p}}  = 
\sum_{q=1}^{m+1} \sum_{\lambda^{F}\in P\left( \bigcup_{a=1}^{m+1}\lbrace a\rbrace ,q\right)} \sigma^{(q)}\left( \sum_{i=1}^{h_{F-1}}\omega_{F}^{\left(j_{F},i\right)}N_{F-1}^{(i)}(x,t)+b_{F}^{\left(j_{F}\right)}\right) \prod_{r=1}^{q} \frac{ \sum_{i=1}^{h_{F-1}} \omega_{F}^{\left( j_{F},i\right)} \partial^{\left\vert \lambda_{r}^{F}\right\vert} N_{F-1}^{(i)}(x,t)}{\prod_{s\in\lambda^{F}_{r}}\partial v_{s}}  , \\
& \left\vert\frac{\partial^{m+1} U_{T}(x,t)}{\partial\omega_{1}^{(j,1)}\partial x^{p}\partial t^{m-p}}\right\vert = \left\vert \sum_{i=1}^{h_{F}}\omega_{F+1}^{(1,i)} \frac{\partial^{m+1}N_{F}^{(i)}(x,t) }{\partial\omega_{1}^{(j,1)}\partial x^{p}\partial t^{m-p}} \right\vert ,
\end{split} \right.
\end{equation*}
where $v_{1}=\ldots=v_{p}=x$, $v_{p+1}=\ldots=v_{m}=t$ and $v_{m+1}=\omega_{1}^{(j,1)}$, for $j_{k}=1,\ldots ,h_{k}$ and $k=2,\ldots ,F-1$.
\begin{equation*}
\left\lbrace \begin{split}
&  \frac{\partial^{\left\vert \lambda_{l}^{2}\right\vert} N_{1}^{\left(j_{1}\right)}(x,t)}{\prod_{s\in\lambda_{l}^{2}}\partial y_{s}} =   \sigma^{\left(\left\vert \lambda_{l}^{2}\right\vert\right)}\left( \omega_{1}^{(j,1)}x+\omega_{1}^{(j,2)}t+b_{1}^{(j)}\right) \prod_{s\in\lambda_{l}^{2}} \left(\delta_{(m+1)s} t +\sum_{i=p+1}^{m} \delta_{is} \omega_{1}^{(j,2)} + \sum_{i=1}^{p} \delta_{is}\omega_{1}^{(j,1)} \right)  , \\
&  \frac{ \partial^{\left\vert \lambda_{l}^{k+1}\right\vert} N_{k}^{\left(j_{k}\right)}(x,t)}{\prod_{s\in\lambda_{l}^{k+1}}\partial y_{s}}  =  \sum_{q=1}^{\left\vert \lambda_{l}^{k+1}\right\vert} \sum_{\lambda^{k}\in P\left( \lambda_{l}^{k+1},q\right)} \sigma^{(q)}\left( \sum_{i=1}^{h_{k-1}}\omega_{k}^{\left(j_{k},i\right)}N_{k-1}^{(i)}(x,t)+b_{k}^{\left(j_{k}\right)}\right) \prod_{r=1}^{q} \frac{ \sum_{i=1}^{h_{k-1}} \omega_{k}^{\left( j_{k},i\right)} \partial^{\left\vert \lambda_{r}^{k}\right\vert} N_{k-1}^{(i)}(x,t)}{\prod_{s\in\lambda_{r}^{k}}\partial y_{s}}  , \\
&  \frac{\partial^{m+1} N_{F}^{\left(j_{F}\right)}(x,t)}{\partial\omega_{1}^{(j,2)}\partial x^{p}\partial t^{m-p}}  =  
\sum_{q=1}^{m+1} \sum_{\lambda^{F}\in P\left( \bigcup_{a=1}^{m+1}\lbrace a\rbrace ,q\right)} \sigma^{(q)}\left( \sum_{i=1}^{h_{F-1}}\omega_{F}^{\left(j_{F},i\right)}N_{F-1}^{(i)}(x,t)+b_{F}^{\left(j_{F}\right)}\right) \prod_{r=1}^{q} \frac{ \sum_{i=1}^{h_{F-1}} \omega_{F}^{\left( j_{F},i\right)} \partial^{\left\vert \lambda_{r}^{F}\right\vert} N_{F-1}^{(i)}(x,t)}{\prod_{s\in\lambda^{F}_{r}}\partial y_{s}}  , \\
& \left\vert\frac{\partial^{m+1} U_{T}(x,t)}{\partial\omega_{1}^{(j,2)}\partial x^{p}\partial t^{m-p}}\right\vert = \left\vert \sum_{i=1}^{h_{F}}\omega_{F+1}^{(1,i)} \frac{\partial^{m+1}N_{F}^{(i)}(x,t) }{\partial\omega_{1}^{(j,2)}\partial x^{p}\partial t^{m-p}} \right\vert ,
\end{split} \right.
\end{equation*}
where $y_{1}=\ldots=y_{p}=x$, $y_{p+1}=\ldots=y_{m}=t$ and $y_{m+1}=\omega_{1}^{(j,2)}$, for $j_{k}=1,\ldots ,h_{k}$ and $k=2,\ldots ,F-1$.
\begin{equation*}
\left\lbrace \begin{split}
& \frac{\partial^{\left\vert \lambda_{l}^{2}\right\vert} N_{1}^{\left(j_{1}\right)}(x,t)}{\prod_{s\in\lambda_{l}^{2}}\partial z_{s}} =   \sigma^{\left(\left\vert \lambda_{l}^{2}\right\vert\right)}\left( \omega_{1}^{(j,1)}x+\omega_{1}^{(j,2)}t+b_{1}^{(j)}\right) \prod_{s\in\lambda_{l}^{2}} \left(\delta_{(m+1)s}  +\sum_{i=p+1}^{m} \delta_{is} \omega_{1}^{(j,2)} + \sum_{i=1}^{p} \delta_{is}\omega_{1}^{(j,1)} \right) , \\
&  \frac{ \partial^{\left\vert \lambda_{l}^{k+1}\right\vert} N_{k}^{\left(j_{k}\right)}(x,t)}{\prod_{s\in\lambda_{l}^{k+1}}\partial z_{s}}  = \sum_{q=1}^{\left\vert \lambda_{l}^{k+1}\right\vert} \sum_{\lambda^{k}\in P\left( \lambda_{l}^{k+1},q\right)} \sigma^{(q)}\left( \sum_{i=1}^{h_{k-1}}\omega_{k}^{\left(j_{k},i\right)}N_{k-1}^{(i)}(x,t)+b_{k}^{\left(j_{k}\right)}\right) \prod_{r=1}^{q} \frac{ \sum_{i=1}^{h_{k-1}} \omega_{k}^{\left( j_{k},i\right)} \partial^{\left\vert \lambda_{r}^{k}\right\vert} N_{k-1}^{(i)}(x,t)}{\prod_{s\in\lambda_{r}^{k}}\partial z_{s}}  , \\
& \frac{\partial^{m+1} N_{F}^{\left(j_{F}\right)}(x,t)}{\partial b_{1}^{(j)}\partial x^{p}\partial t^{m-p}} =  
\sum_{q=1}^{m+1} \sum_{\lambda^{F}\in P\left( \bigcup_{a=1}^{m+1}\lbrace a\rbrace ,q\right)} \sigma^{(q)}\left( \sum_{i=1}^{h_{F-1}}\omega_{F}^{\left(j_{F},i\right)}N_{F-1}^{(i)}(x,t)+b_{F}^{\left(j_{F}\right)}\right) \prod_{r=1}^{q} \frac{ \sum_{i=1}^{h_{F-1}} \omega_{F}^{\left( j_{F},i\right)} \partial^{\left\vert \lambda_{r}^{F}\right\vert} N_{F-1}^{(i)}(x,t)}{\prod_{s\in\lambda^{F}_{r}}\partial z_{s}} , \\
& \left\vert\frac{\partial^{m+1} U_{T}(x,t)}{\partial b_{1}^{(j)}\partial x^{p}\partial t^{m-p}}\right\vert = \left\vert \sum_{i=1}^{h_{F}}\omega_{F+1}^{(1,i)} \frac{\partial^{m+1}N_{F}^{(i)}(x,t) }{\partial b_{1}^{(j)}\partial x^{p}\partial t^{m-p}} \right\vert ,
\end{split} \right.
\end{equation*}
where $z_{1}=\ldots=z_{p}=x$, $z_{p+1}=\ldots=z_{m}=t$ and $z_{m+1}=b_{1}^{(j)}$, for $j_{k}=1,\ldots ,h_{k}$ and $k=2,\ldots ,F-1$.

Since $\left\vert\sigma^{(n)}(x)\right\vert\leq 1$ for all $n\in\mathbb{N}$ and $x\in\mathbb{R}$ and $\vert P(A,q)\vert = S(\vert A\vert, q)$ is finite, if $\Omega_{x}\times\Omega_{t}$ is bounded, from the previous recurrence equations, it follows that
\begin{equation*}
\left\lbrace \begin{split}
& \left\vert \frac{\partial^{\left\vert \lambda_{l}^{2}\right\vert} N_{1}^{\left(j_{1}\right)}(x,t)}{\prod_{s\in\lambda_{l}^{2}}\partial v_{s}} \right\vert \leq \left\vert \sigma^{\left(\left\vert \lambda_{l}^{2}\right\vert\right)}\left( \omega_{1}^{(j,1)}x+\omega_{1}^{(j,2)}t+b_{1}^{(j)}\right) \right\vert \max_{x\in\Omega_{x}} \left\lbrace\prod_{s\in\lambda_{l}^{2}} \left\vert\delta_{(m+1)s} x +\sum_{i=p+1}^{m} \delta_{is} \omega_{1}^{(j,2)} + \sum_{i=1}^{p} \delta_{is}\omega_{1}^{(j,1)} \right\vert \right\rbrace  , \\
& \left\vert\frac{ \partial^{\left\vert \lambda_{l}^{k+1}\right\vert} N_{k}^{\left(j_{k}\right)}(x,t)}{\prod_{s\in\lambda_{l}^{k+1}}\partial v_{s}}  \right\vert \leq  \sum_{q=1}^{\left\vert \lambda_{l}^{k+1}\right\vert} \sum_{\lambda^{k}\in P\left( \lambda_{l}^{k+1},q\right)} \prod_{r=1}^{q} 
\sum_{i=1}^{h_{k-1}} \left\vert\omega_{k}^{\left( j_{k},i\right)}\right\vert \left\vert\frac{ \partial^{\left\vert \lambda_{r}^{k}\right\vert} N_{k-1}^{(i)}(x,t)}{\prod_{s\in\lambda_{r}^{k}}\partial v_{s}} \right\vert, \\
& \left\vert \frac{\partial^{m+1} N_{F}^{\left(j_{F}\right)}(x,t)}{\partial\omega_{1}^{(j,1)}\partial x^{p}\partial t^{m-p}} \right\vert \leq 
\sum_{q=1}^{m+1} \sum_{\lambda^{F}\in P\left( \bigcup_{a=1}^{m+1}\lbrace a\rbrace ,q\right)} \prod_{r=1}^{q} \sum_{i=1}^{h_{F-1}} \left\vert\omega_{F}^{\left( j_{F},i\right)} \right\vert \left\vert\frac{ \partial^{\left\vert \lambda_{r}^{F}\right\vert} N_{F-1}^{(i)}(x,t)}{\prod_{s\in\lambda^{F}_{r}}\partial v_{s}} \right\vert , \\
& \left\vert\frac{\partial^{m+1} U_{T}(x,t)}{\partial\omega_{1}^{(j,1)}\partial x^{p}\partial t^{m-p}}\right\vert \leq \sum_{i=1}^{h_{F}} \left\vert \omega_{F+1}^{(1,i)} \right\vert \left\vert \frac{\partial^{m+1}N_{F}^{(i)}(x,t) }{\partial\omega_{1}^{(j,1)}\partial x^{p}\partial t^{m-p}} \right\vert ,
\end{split} \right.
\end{equation*}
for $j_{k}=1,\ldots ,h_{k}$ and $k=2,\ldots ,F-1$.
\begin{equation*}
\left\lbrace \begin{split}
& \left\vert \frac{\partial^{\left\vert \lambda_{l}^{2}\right\vert} N_{1}^{\left(j_{1}\right)}(x,t)}{\prod_{s\in\lambda_{l}^{2}}\partial y_{s}} \right\vert \leq \left\vert \sigma^{\left(\left\vert \lambda_{l}^{2}\right\vert\right)}\left( \omega_{1}^{(j,1)}x+\omega_{1}^{(j,2)}t+b_{1}^{(j)}\right) \right\vert \max_{t\in\Omega_{t}} \left\lbrace\prod_{s\in\lambda_{l}^{2}} \left\vert\delta_{(m+1)s} t +\sum_{i=p+1}^{m} \delta_{is} \omega_{1}^{(j,2)} + \sum_{i=1}^{p} \delta_{is}\omega_{1}^{(j,1)} \right\vert \right\rbrace  , \\
& \left\vert\frac{ \partial^{\left\vert \lambda_{l}^{k+1}\right\vert} N_{k}^{\left(j_{k}\right)}(x,t)}{\prod_{s\in\lambda_{l}^{k+1}}\partial y_{s}}  \right\vert \leq  \sum_{q=1}^{\left\vert \lambda_{l}^{k+1}\right\vert} \sum_{\lambda^{k}\in P\left( \lambda_{l}^{k+1},q\right)} \prod_{r=1}^{q} 
\sum_{i=1}^{h_{k-1}} \left\vert\omega_{k}^{\left( j_{k},i\right)}\right\vert \left\vert\frac{ \partial^{\left\vert \lambda_{r}^{k}\right\vert} N_{k-1}^{(i)}(x,t)}{\prod_{s\in\lambda_{r}^{k}}\partial y_{s}} \right\vert ,\\
& \left\vert \frac{\partial^{m+1} N_{F}^{\left(j_{F}\right)}(x,t)}{\partial\omega_{1}^{(j,2)}\partial x^{p}\partial t^{m-p}} \right\vert \leq 
\sum_{q=1}^{m+1} \sum_{\lambda^{F}\in P\left( \bigcup_{a=1}^{m+1}\lbrace a\rbrace ,q\right)} \prod_{r=1}^{q} \sum_{i=1}^{h_{F-1}} \left\vert\omega_{F}^{\left( j_{F},i\right)} \right\vert \left\vert\frac{ \partial^{\left\vert \lambda_{r}^{F}\right\vert} N_{F-1}^{(i)}(x,t)}{\prod_{s\in\lambda^{F}_{r}}\partial y_{s}} \right\vert , \\
& \left\vert\frac{\partial^{m+1} U_{T}(x,t)}{\partial\omega_{1}^{(j,2)}\partial x^{p}\partial t^{m-p}}\right\vert \leq \sum_{i=1}^{h_{F}} \left\vert \omega_{F+1}^{(1,i)} \right\vert \left\vert \frac{\partial^{m+1}N_{F}^{(i)}(x,t) }{\partial\omega_{1}^{(j,2)}\partial x^{p}\partial t^{m-p}} \right\vert ,
\end{split} \right.
\end{equation*}
for $j_{k}=1,\ldots ,h_{k}$ and $k=2,\ldots ,F-1$.
\begin{equation*}
\left\lbrace \begin{split}
& \left\vert \frac{\partial^{\left\vert \lambda_{l}^{2}\right\vert} N_{1}^{\left(j_{1}\right)}(x,t)}{\prod_{s\in\lambda_{l}^{2}}\partial z_{s}} \right\vert \leq \left\vert \sigma^{\left(\left\vert \lambda_{l}^{2}\right\vert\right)}\left( \omega_{1}^{(j,1)}x+\omega_{1}^{(j,2)}t+b_{1}^{(j)}\right) \right\vert \prod_{s\in\lambda_{l}^{2}} \left\vert\delta_{(m+1)s}  +\sum_{i=p+1}^{m} \delta_{is} \omega_{1}^{(j,2)} + \sum_{i=1}^{p} \delta_{is}\omega_{1}^{(j,1)} \right\vert   , \\
& \left\vert\frac{ \partial^{\left\vert \lambda_{l}^{k+1}\right\vert} N_{k}^{\left(j_{k}\right)}(x,t)}{\prod_{s\in\lambda_{l}^{k+1}}\partial z_{s}}  \right\vert \leq  \sum_{q=1}^{\left\vert \lambda_{l}^{k+1}\right\vert} \sum_{\lambda^{k}\in P\left( \lambda_{l}^{k+1},q\right)} \prod_{r=1}^{q} 
\sum_{i=1}^{h_{k-1}} \left\vert\omega_{k}^{\left( j_{k},i\right)}\right\vert \left\vert\frac{ \partial^{\left\vert \lambda_{r}^{k}\right\vert} N_{k-1}^{(i)}(x,t)}{\prod_{s\in\lambda_{r}^{k}}\partial z_{s}} \right\vert , \\
& \left\vert \frac{\partial^{m+1} N_{F}^{\left(j_{F}\right)}(x,t)}{\partial b_{1}^{(j)}\partial x^{p}\partial t^{m-p}} \right\vert \leq 
\sum_{q=1}^{m+1} \sum_{\lambda^{F}\in P\left( \bigcup_{a=1}^{m+1}\lbrace a\rbrace ,q\right)} \prod_{r=1}^{q} \sum_{i=1}^{h_{F-1}} \left\vert\omega_{F}^{\left( j_{F},i\right)} \right\vert \left\vert\frac{ \partial^{\left\vert \lambda_{r}^{F}\right\vert} N_{F-1}^{(i)}(x,t)}{\prod_{s\in\lambda^{F}_{r}}\partial z_{s}} \right\vert , \\
& \left\vert\frac{\partial^{m+1} U_{T}(x,t)}{\partial b_{1}^{(j)}\partial x^{p}\partial t^{m-p}}\right\vert \leq \sum_{i=1}^{h_{F}} \left\vert \omega_{F+1}^{(1,i)} \right\vert \left\vert \frac{\partial^{m+1}N_{F}^{(i)}(x,t) }{\partial b_{1}^{(j)}\partial x^{p}\partial t^{m-p}} \right\vert ,
\end{split} \right.
\end{equation*}
for $j_{k}=1,\ldots ,h_{k}$ and $k=2,\ldots ,F-1$.

Define $M_{(j,1)}^{m+1}$, $M_{(j,2)}^{m+1}$ and $ M_{(j,3)}^{m+1}$ for all $m=0,\ldots ,n$ and $j=1,\ldots ,h_{1}$, given by the following recurrence equations
\begin{equation*}
\left\lbrace \begin{split}
& M_{(j,1)}^{\left( m+1 ,\left\vert \lambda_{l}^{2}\right\vert\right)} = \max_{x\in \Omega_{x}}\left\lbrace \prod_{s\in\lambda_{l}^{2}} \left\vert\delta_{(m+1)s} x +\sum_{i=p+1}^{m} \delta_{is} \omega_{1}^{(j,2)} + \sum_{i=1}^{p} \delta_{is}\omega_{1}^{(j,1)} \right\vert \right\rbrace, \\
& M_{(j,1)}^{\left( m+1 ,\left\vert \lambda_{l}^{k+1}\right\vert\right)} =  \sum_{q=1}^{\left\vert \lambda_{l}^{k+1}\right\vert} \sum_{\lambda^{k}\in P\left( \lambda_{l}^{k+1},q\right)} \prod_{r=1}^{q} \sum_{i=1}^{h_{k-1}} \left\vert\omega_{k}^{\left( j_{k},i\right)}\right\vert  M_{(j,1)}^{\left( m+1 ,\left\vert \lambda_{r}^{k}\right\vert\right)}, \\
& M_{(j,1)}^{(m+1,F)} = \sum_{q=1}^{m+1} \sum_{\lambda^{F}\in P\left( \bigcup_{a=1}^{m+1}\lbrace a\rbrace ,q\right)} \prod_{r=1}^{q} \sum_{i=1}^{h_{F-1}} \left\vert\omega_{F}^{\left( j_{F},i\right)}\right\vert M_{(j,1)}^{\left( m+1,\left\vert \lambda_{r}^{F}\right\vert\right)} ,\\
& M_{(j,1)}^{m+1} =\sum_{i=1}^{h_{F}} \left\vert \omega_{F+1}^{(1,i)}\right\vert M_{(j,1)}^{(m+1,F)} ,
\end{split} \right.
\end{equation*}
for $j_{k}=1,\ldots ,h_{k}$ and $k=2,\ldots ,F-1$.
\begin{equation*}
\left\lbrace \begin{split}
& M_{(j,2)}^{\left( m+1,\left\vert \lambda_{l}^{2}\right\vert\right)} = \max_{t\in \Omega_{t}}\left\lbrace \prod_{s\in\lambda_{l}^{2}} \left\vert\delta_{(m+1)s} t +\sum_{i=p+1}^{m} \delta_{is} \omega_{1}^{(j,2)} + \sum_{i=1}^{p} \delta_{is}\omega_{1}^{(j,1)} \right\vert \right\rbrace, \\
& M_{(j,2)}^{\left( m+1,\left\vert \lambda_{l}^{k+1}\right\vert\right)} =  \sum_{q=1}^{\left\vert \lambda_{l}^{k+1}\right\vert} \sum_{\lambda^{k}\in P\left( \lambda_{l}^{k+1},q\right)} \prod_{r=1}^{q} \sum_{i=1}^{h_{k-1}} \left\vert\omega_{k}^{\left( j_{k},i\right)}\right\vert  M_{(j,2)}^{\left( m+1 ,\left\vert \lambda_{r}^{k}\right\vert\right)} , \\
& M_{(j,2)}^{(m+1,F)} = \sum_{q=1}^{m+1} \sum_{\lambda^{F}\in P\left( \bigcup_{a=1}^{m+1}\lbrace a\rbrace ,q\right)} \prod_{r=1}^{q} \sum_{i=1}^{h_{F-1}} \left\vert\omega_{F}^{\left( j_{F},i\right)}\right\vert M_{(j,2)}^{\left(m+1,\left\vert \lambda_{r}^{F}\right\vert\right)} ,\\
& M_{(j,2)}^{m+1} =\sum_{i=1}^{h_{F}} \left\vert \omega_{F+1}^{(1,i)}\right\vert M_{(j,2)}^{(m+1,F)} ,
\end{split} \right.
\end{equation*}
for $j_{k}=1,\ldots ,h_{k}$ and $k=2,\ldots ,F-1$.
\begin{equation*}
\left\lbrace \begin{split}
& M_{(j,3)}^{\left( m+1,\left\vert \lambda_{l}^{2}\right\vert\right)} = \prod_{s\in\lambda_{l}^{2}} \left\vert\delta_{(m+1)s} +\sum_{i=p+1}^{m} \delta_{is} \omega_{1}^{(j,2)} + \sum_{i=1}^{p} \delta_{is}\omega_{1}^{(j,1)} \right\vert, \\
& M_{(j,3)}^{\left( m+1,\left\vert \lambda_{l}^{k+1}\right\vert\right)} =  \sum_{q=1}^{\left\vert \lambda_{l}^{k+1}\right\vert} \sum_{\lambda^{k}\in P\left( \lambda_{l}^{k+1},q\right)} \prod_{r=1}^{q} \sum_{i=1}^{h_{k-1}} \left\vert\omega_{k}^{\left( j_{k},i\right)}\right\vert  M_{(j,3)}^{\left( m+1,\left\vert \lambda_{r}^{k}\right\vert\right)} , \\
& M_{(j,3)}^{(m+1,F)} = \sum_{q=1}^{m+1} \sum_{\lambda^{F}\in P\left( \bigcup_{a=1}^{m+1}\lbrace a\rbrace ,q\right)} \prod_{r=1}^{q} \sum_{i=1}^{h_{F-1}} \left\vert\omega_{F}^{\left( j_{F},i\right)}\right\vert M_{(j,3)}^{\left( m+1,\left\vert \lambda_{r}^{F}\right\vert\right)} ,\\
& M_{(j,3)}^{m+1} =\sum_{i=1}^{h_{F}} \left\vert \omega_{F+1}^{(1,i)}\right\vert M_{(j,3)}^{(m+1,F)} ,
\end{split} \right.
\end{equation*}
for $j_{k}=1,\ldots ,h_{k}$ and $k=2,\ldots ,F-1$.

Thus, taking
\begin{equation*}
M= \max_{\substack{i\in\lbrace 1,2,3\rbrace,\\ j\in\lbrace 1,\ldots,h_{1}\rbrace,\\ m\in\lbrace 0,\ldots ,n\rbrace}} \left\lbrace  \frac{1}{M_{(j,i)}^{m+1}}\right\rbrace , 
\end{equation*}
Lemma $\ref{Propic}$ implies that there is $\delta_{\epsilon_{j}}>0$ such that if $\left\vert \omega_{1}^{(j,1)}x+\omega_{1}^{(j,2)}t+ b_{1}^{(j)}\right\vert > \delta_{\epsilon_{j}}$ then $\sigma ^{(m+1)}\left( \omega_{1}^{(j,1)}x+\omega_{1}^{(j,2)}t+ b_{1}^{(j)}\right) < \epsilon M$ for all $m=0,\ldots,n$. Therefore, for all $(x,t)\notin \left\lbrace (x,t)\in\Omega_{x}\times\Omega_{t} : \left\vert \omega_{1}^{(j,1)}x+\omega_{1}^{(j,2)}t+ b_{1}^{(j)}\right\vert < \delta_{\epsilon}\right\rbrace$, with $\delta_{\epsilon}=\min_{j\in\lbrace 1,\ldots,h_{1}\rbrace}\left\lbrace \delta_{\epsilon_{j}}\right\rbrace$, we have
\begin{align*}
\left\vert\frac{\partial^{m+1} U_{T}(x,t)}{\partial\omega_{1}^{(j,1)}\partial x^{p}\partial t^{m-p}}\right\vert & < \epsilon M M_{(j,1)}^{m+1} \leq \epsilon ,\\
\left\vert\frac{\partial^{m+1} U_{T}(x,t)}{\partial\omega_{1}^{(j,2)}\partial x^{p}\partial t^{m-p}}\right\vert & < \epsilon M M_{(j,2)}^{m+1}\leq \epsilon,\\
\left\vert\frac{\partial^{m+1} U_{T}(x,t)}{\partial b_{1}^{(j)}\partial x^{p}\partial t^{m-p}}\right\vert & < \epsilon M M_{(j,3)}^{m+1} \leq \epsilon,
\end{align*}
for all $0\leq p\leq m\leq n$ and $j=1,\ldots ,h_{1}$.
\end{proof}

Note that for all $k=2,\ldots, F+1$, it is possible to express $\tfrac{\partial^{m+1} U_{T}(x,t)}{\partial\omega_{k}^{(j,i)}\partial x^{p}\partial t^{m-p}}$ and $\tfrac{\partial^{m+1} U_{T}(x,t)}{\partial b_{k}^{(j)}\partial x^{p}\partial t^{m-p}}$ in a similar way as was done in the proof of Theorem \ref{Coric}, whenever $p\geq 1$ or $m\geq 1$, by deleting the terms that have $\delta_{(m+1)s}$ as a factor. By doing the same in the rest of the proof of Theorem \ref{Coric}, that is, rewriting the inequalities for $\left\vert\tfrac{\partial^{m+1} U_{T}(x,t)}{\partial\omega_{k}^{(j,i)}\partial x^{p}\partial t^{m-p}}\right\vert$ and $\left\vert\tfrac{\partial^{m+1} U_{T}(x,t)}{\partial b_{k}^{(j)}\partial x^{p}\partial t^{m-p}}\right\vert$, and the expressions for $M_{(j,1)}^{(m+1,k)}$ for each $k\in\lbrace 2,\ldots ,F+1\rbrace$, the following corollary is obtained.
\begin{corollary}
\label{Cortheo1}
Let $U_{T}(x,t)$ be given by Eq. $(\ref{NNs})$, $\epsilon \ll 1$ and $n\in\mathbb{N}$. For any $j\leq h_{1}$, there exists $\delta_{\epsilon_{j}}>0$ such that for all
\begin{equation*}
\label{dgzcor1}
(x,t)\notin \left\lbrace (x,t)\in\Omega_{x}\times\Omega_{t} : \left\vert \omega_{1}^{(j,1)}x+\omega_{1}^{(j,2)}t+ b_{1}^{(j)}\right\vert < \delta_{\epsilon_{j}}\right\rbrace ,
\end{equation*}
it holds that 
\begin{equation}
\label{CorTheo1}
\left\vert\frac{\partial^{m+1} U_{T}(x,t)}{\partial\omega_{k}^{(j,i)}\partial x^{p}\partial t^{m-p}}\right\vert <\epsilon \ \ \ \mathrm{and} \ \ \ \left\vert\frac{\partial^{m+1} U_{T}(x,t)}{\partial b_{k}^{(j)}\partial x^{p}\partial t^{m-p}}\right\vert <\epsilon ,
\end{equation}
for $i=1,2$, $k=2,\ldots, F+1$, $0\leq p \leq m$ and $1\leq m\leq n$.
\end{corollary}

Theorem $\ref{Coric}$ guarantees the existence of a region where all the partial derivatives $\tfrac{\partial^{m+1} U_{T}(x,t)}{\partial\omega_{1}^{(j,i)}\partial x^{p}\partial t^{m-p}}$ and $\tfrac{\partial^{m+1} U_{T}(x,t)}{\partial b_{1}^{(j)}\partial x^{p}\partial t^{m-p}}$ are negligible for $i=1,2$ and $0\leq p\leq m \leq n$. Meanwhile, Corollary $\ref{Cortheo1}$ guarantees, under certain conditions for the PDE, a similar result for all the derivatives $\tfrac{\partial^{m+1} U_{T}(x,t)}{\partial\omega_{k}^{(j,i)}\partial x^{p}\partial t^{m-p}}$ and $\tfrac{\partial^{m+1} U_{T}(x,t)}{\partial b_{k}^{(j)}\partial x^{p}\partial t^{m-p}}$ for $i=1,2$, $k=2,\ldots, F+1$, $0\leq p \leq m$ and $1\leq m\leq n$. These results allow a comparison to be made between the gradient of $Loss_{Total}$ in a sample $S$ over the whole domain of the PDE and the gradient of $Loss_{Total}$ in the sample $S-S_{\epsilon}$, which results from removing all the points in $S$ that are within the diminishing gradient zones of $\sigma\left(\omega_{1}^{(j,1)}x+\omega_{1}^{(j,2)}t+ b_{1}^{(j)}\right)$ for all $j=1,\ldots h_{1}$, stated in Theorem $\ref{theopde}$. The reason for making this comparison is to show that removing the points in the sample in which the gradient of $Loss_{Total}$ is negligible, represents an improvement in the optimization process. This is because they have the effect of reducing the step size with which the parameters are updated, slowing down the optimization process.

\begin{theor}
\label{theopde}
Let $\mathcal{F}\left( U,\nabla U,\ldots\right)=0$ a PDE with initial and boundary conditions over $\overline{\Omega_{x}\times\Omega_{t}}$, $\epsilon \ll 1$, $S$ a sample of points within the domain of $\mathcal{F}$, and $U_{T}(x,t)$ given by Eq. $(\ref{NNs})$. There are $\delta_{\epsilon_{j}}>0$ such that
\begin{equation}
\label{Theo2}
\left\vert\frac{\partial Loss_{Total}}{\partial\omega_{1}^{(j,i)}} \right\vert_{S}<\left\vert\frac{\partial Loss_{Total}}{\partial\omega_{1}^{(j,i)}} \right\vert_{S-S_{\epsilon}} \ \ \ and \ \ \ \left\vert\frac{\partial Loss_{Total}}{\partial b_{1}^{(j)}} \right\vert_{S}<\left\vert\frac{\partial Loss_{Total}}{\partial b_{1}^{(j)}} \right\vert_{S-S_{\epsilon}}
\end{equation}
for $i=1,2$ and $j=1,\ldots ,h_{1}$, where $S_{\epsilon}=S_{\epsilon_{IC}}\cup S_{\epsilon_{LBC}}\cup S_{\epsilon_{RBC}}\cup S_{\epsilon_{PDE}}$ is a set of the form
\begin{equation}
\label{Dgztheo2}
S_{\epsilon}= \left\lbrace (x,t)\in S^{\prime}\subset \overline{\Omega_{x}\times\Omega_{t}} : \left\vert \omega_{1}^{(j,1)}x+\omega_{1}^{(j,2)}t+ b_{1}^{(j)}\right\vert >\delta^{\prime}_{\epsilon_{j}} \ \forall j\leq h_{1}\right\rbrace ,
\end{equation}
with $S^{\prime}\in\left\lbrace S_{IC},S_{LBC},S_{RBC},S_{PDE}\right\rbrace$.
\end{theor}

\begin{proof}
Let $\mathcal{F}\left( U,\nabla U,\ldots\right)=0$ a PDE over $\Omega_{x}\times\Omega_{t}$, with initial condition $U_{0}(x)=U\left( x,t_{0}\right)$ and boundary conditions. For simplicity, we will denote $\mathcal{F}(U,\nabla U,\ldots)$ as $\mathcal{F}(U)$. Let $\epsilon \ll 1$, a sample of points $S$ within $\overline{\Omega_{x}\times\Omega_{t}}$ and $U_{T}(x,t)$ given by Eq. $(\ref{NNs})$. Considering the partition of $S$ given by Eq. $(\ref{partS})$, from Eq. $(\ref{Loss})$ it follows that
\begin{equation}
\begin{split}
\label{LossT}
    \left\vert\frac{\partial Loss_{Total}}{\partial\omega_{1}^{(j,i)}} \right\vert & = \left\vert\frac{\partial Loss_{IC}}{\partial\omega_{1}^{(j,i)}} +\frac{\partial Loss_{LBC}}{\partial\omega_{1}^{(j,i)}} + \frac{\partial Loss_{RBC}}{\partial\omega_{1}^{(j,i)}} + \frac{\partial Loss_{PDE}}{\partial\omega_{1}^{(j,i)}} \right\vert ,\\
    \left\vert\frac{\partial Loss_{Total}}{\partial b_{1}^{(j)}} \right\vert & = \left\vert\frac{\partial Loss_{IC}}{\partial b_{1}^{(j)}} + \frac{\partial Loss_{LBC}}{\partial b_{1}^{(j)}} + \frac{\partial Loss_{RBC}}{\partial b_{1}^{(j)}} + \frac{\partial Loss_{PDE}}{\partial b_{1}^{(j)}} \right\vert .
\end{split}
\end{equation}

The partial derivatives of $Loss_{IC}$ involved in Eq. $(\ref{LossT})$ can be expressed as
\begin{equation}
\label{theo_IC}
\begin{split}
\frac{\partial Loss_{IC}}{\partial\omega_{1}^{(j,i)}} & = \frac{2}{\left\vert S_{IC}\right\vert} \sum_{\left(x,t_{0}\right)\in S_{IC}} \left( U_{T}\left( x,t_{0}\right)-U_{0}(x)\right) \frac{\partial U_{T}\left(x,t_{0}\right)}{\partial\omega_{1}^{(j,i)}}, \\
\frac{\partial Loss_{IC}}{\partial b_{1}^{(j)}} & = \frac{2}{\left\vert S_{IC}\right\vert} \sum_{\left(x,t_{0}\right)\in S_{IC}} \left( U_{T}\left( x,t_{0}\right)-U_{0}(x)\right) \frac{\partial U_{T}\left(x,t_{0}\right)}{\partial b_{1}^{(j)}} ,
\end{split}
\end{equation}
 
for $i=1,2$ and $j=1,\ldots ,h_{1}$. Taking
\begin{equation}
\label{E_IC}
\epsilon_{IC} = \frac{\left\vert S_{IC}\right\vert \epsilon}{8\sum_{\left(x,t_{0}\right)\in S_{IC}}\left\vert U_{T}\left( x,t_{0}\right)-U_{0}(x)\right\vert},
\end{equation}
Theorem \ref{Coric} guarantees that there exists $\delta_{\epsilon_{IC_{j}}}>0$ such that for all
\begin{equation*}
\left( x,t_{0}\right)\notin D_{\epsilon_{IC}}=\left\lbrace \left( x,t_{0}\right)\in\Omega_{x}\times\left\lbrace t_{0}\right\rbrace : \left\vert \omega_{1}^{(j,1)}x+\omega_{1}^{(j,2)}t_{0}+ b_{1}^{(j)}\right\vert < \delta_{\epsilon_{IC_{j}}}, j=1,\ldots,h_{1}\right\rbrace ,
\end{equation*}
it holds that 
\begin{equation}
\label{PICtheo2}
\left\vert\frac{\partial U_{T}\left(x,t_{0}\right)}{\partial\omega_{1}^{(j,i)}}\right\vert<\epsilon_{IC} \ \mathrm{and} \ \left\vert\frac{\partial U_{T}\left(x,t_{0}\right)}{\partial b_{1}^{(j)}}\right\vert<\epsilon_{IC}, 
\end{equation}
for $i=1,2$ and $j=1,\ldots ,h_{1}$. Therefore, the absolute value of all partial derivatives $\tfrac{\partial U_{T}\left(x,t_{0}\right)}{\partial\omega_{1}^{(j,i)}}$ and $\tfrac{\partial U_{T}\left(x,t_{0}\right)}{\partial b_{1}^{(j)}}$ in Eq. $(\ref{theo_IC})$ are bounded by $\epsilon_{IC}$ for all $\left( x,t_{0}\right)\in S_{\epsilon_{IC}}=S_{IC}\cap D_{\epsilon_{IC}}^{c}$.

Similarly, it is possible to bound the absolute value of all the mixed partial derivatives of $U_{T}(x,t)$ involved in $\frac{\partial Loss_{PDE}}{\partial\omega_{1}^{(j,i)}}$ and $\frac{\partial Loss_{PDE}}{\partial b_{1}^{(j)}}$ for a subset of $S_{PDE}$. Assuming $\mathcal{F}\left( U(x,t)\right)$ is a PDE of order $n$, the partial derivatives of $Loss_{PDE}$ are
\begin{equation}
\label{theo_PDE}
\begin{split}
    \frac{\partial Loss_{PDE}}{\partial\omega_{1}^{(j,i)}} & = \frac{2}{\left\vert S_{PDE}\right\vert} \sum_{(x,t)\in S_{PDE}}  \mathcal{F}\left( U_{T}(x,t)\right) \frac{\partial \mathcal{F}\left( U_{T}(x,t)\right)}{\partial\omega_{1}^{(j,i)}}, \\
    \frac{\partial Loss_{PDE}}{\partial b_{1}^{(j)}} & = \frac{2}{\left\vert S_{PDE}\right\vert}\sum_{(x,t)\in S_{PDE}} \mathcal{F}\left( U_{T}(x,t)\right) \frac{\partial \mathcal{F}\left( U_{T}(x,t)\right)}{\partial b_{1}^{(j)}} ,
\end{split}
\end{equation}

for $i=1,2$ and $j=1,\ldots ,h_{1}$. Every term in $\tfrac{\partial \mathcal{F}\left( U_{T}(x,t)\right)}{\partial\omega_{1}^{(j,i)}}$ and $\tfrac{\partial \mathcal{F}\left( U_{T}(x,t)\right)}{\partial b_{1}^{(j)}}$ contains a factor of the form $\tfrac{\partial^{m+1} U_{T}(x,t)}{\partial\omega_{1}^{(j,i)}\partial x^{p}\partial t^{m-p}}$ or $\tfrac{\partial^{m+1} U_{T}(x,t)}{\partial b_{1}^{(j)}\partial x^{p}\partial t^{m-p}}$ for some $m\in\lbrace 0,\ldots,n\rbrace$ and $0\leq p\leq m$, for $i=1,2$. Define $\mathcal{G}\left( U_{T}(x,t) \right)$ as the expression resulting from considering the sum of the absolute values of the terms obtained by removing exactly one of the factors $\tfrac{\partial^{m+1} U_{T}(x,t)}{\partial\omega_{1}^{(j,i)}\partial x^{p}\partial t^{m-p}}$ or $\tfrac{\partial^{m+1} U_{T}(x,t)}{\partial b_{1}^{(j)}\partial x^{p}\partial t^{m-p}}$ from each term in $\tfrac{\partial \mathcal{F}\left( U_{T}(x,t)\right)}{\partial\omega_{1}^{(j,i)}}$ or $\tfrac{\partial \mathcal{F}\left( U_{T}(x,t)\right)}{\partial b_{1}^{(j)}}$, respectively. If $\mathcal{F}\left( U_{T}(x,t)\right)\neq 0$, consider 
\begin{equation}
\label{EP_PDE}
\epsilon_{PDE}^{\prime} = \frac{\left\vert S_{PDE}\right\vert \epsilon}{8 \sum_{(x,t)\in S_{PDE}} \left\vert \mathcal{F}\left( U_{T}(x,t) \right)\right\vert \max_{(x,t)\in S_{PDE}}\left\lbrace \mathcal{G}\left( U_{T}(x,t) \right)\right\rbrace },
\end{equation}
Theorem \ref{Coric} allows to guarantee that there are $\delta_{\epsilon_{PDE_{j}}}>0$ such that for all
\begin{equation*}
(x,t)\notin D_{\epsilon_{PDE}}=\left\lbrace (x,t)\in\Omega_{x}\times\Omega_{t} : \left\vert \omega_{1}^{(j,1)}x+\omega_{1}^{(j,2)}t+ b_{1}^{(j)}\right\vert < \delta_{\epsilon_{PDE_{j}}},j=1,\ldots,h_{1}\right\rbrace ,
\end{equation*}
it holds that
\begin{equation*}
\left\vert\frac{\partial^{m+1} U_{T}(x,t)}{\partial\omega_{1}^{(j,i)}\partial x^{p}\partial t^{m-p}}\right\vert<\epsilon^{\prime}_{PDE} \ \mathrm{and} \ \left\vert\frac{\partial^{m+1} U_{T}(x,t)}{\partial b_{1}^{(j)}\partial x^{p}\partial t^{m-p}}\right\vert<\epsilon^{\prime}_{PDE} ,   
\end{equation*}
for all $(x,t)\in S_{\epsilon_{PDE}}=S_{PDE}\cap D_{\epsilon_{PDE}}$, for $i=1,2$ and $j=1,\ldots ,h_{1}$. Therefore, consider 
\begin{equation}
\label{E_PDE}
\epsilon_{PDE} = \epsilon_{PDE}^{\prime} \max_{(x,t)\in S_{PDE}}\left\lbrace \mathcal{G}\left( U_{T}(x,t) \right)\right\rbrace = \frac{\left\vert S_{PDE}\right\vert \epsilon}{8 \sum_{(x,t)\in S_{PDE}} \left\vert \mathcal{F}\left( U_{T}(x,t) \right)\right\vert },
\end{equation}
implies that
\begin{equation}
\label{PPDEtheo2}
\begin{split}
    \left\vert\frac{\partial \mathcal{F}\left( U_{T}(x,t)\right)}{\partial\omega_{1}^{(j,i)}}\right\vert & < \max_{\substack{(x,t)\in S_{PDE},\\ m\in\lbrace 0,\ldots ,n\rbrace,\\ p\in\lbrace 0,\ldots ,m\rbrace}}\left\lbrace \left\vert\frac{\partial^{m+1} U_{T}(x,t)}{\partial\omega_{1}^{(j,i)}\partial x^{p}\partial t^{m-p}}\right\vert\right\rbrace \max_{(x,t)\in S_{PDE}}\left\lbrace \mathcal{G}\left( U_{T}(x,t) \right)\right\rbrace \\
    & < \epsilon_{PDE}^{\prime} \max_{(x,t)\in S_{PDE}}\left\lbrace \mathcal{G}\left( U_{T}(x,t) \right)\right\rbrace \\
    & = \epsilon_{PDE}, \\
    \left\vert\frac{\partial \mathcal{F}\left( U_{T}(x,t)\right)}{\partial b_{1}^{(j)}}\right\vert & < \max_{\substack{(x,t)\in S_{PDE},\\ m\in\lbrace 0,\ldots ,n\rbrace,\\ p\in\lbrace 0,\ldots ,m\rbrace}}\left\lbrace \left\vert\frac{\partial^{m+1} U_{T}(x,t)}{\partial b_{1}^{(j)}\partial x^{p}\partial t^{m-p}}\right\vert\right\rbrace \max_{(x,t)\in S_{PDE}}\left\lbrace \mathcal{G}\left( U_{T}(x,t) \right)\right\rbrace \\
    & < \epsilon_{PDE}^{\prime} \max_{(x,t)\in S_{PDE}}\left\lbrace \mathcal{G}\left( U_{T}(x,t) \right)\right\rbrace \\
    & = \epsilon_{PDE},
\end{split}
\end{equation}
for all $(x,t)\in S_{\epsilon_{PDE}}=S_{PDE}\cap D_{\epsilon_{PDE}}$, for $i=1,2$ and $j=1,\ldots ,h_{1}$.

For Dirichlet boundary conditions, the partial derivatives of $Loss_{LBC}$ and $Loss_{RBC}$ are similar to Eq. $(\ref{theo_IC})$, while for Neumann boundary conditions, the partial derivatives of $Loss_{LBC}$ and $Loss_{RBC}$ are similar to Eq. $(\ref{theo_PDE})$. Therefore, following the same analysis conducted for the derivatives involved in Eq. $(\ref{theo_IC})$ or Eq. $(\ref{theo_PDE})$, a result similar to Eq. $(\ref{PICtheo2})$ or Eq. $(\ref{PPDEtheo2})$ for the boundary conditions is achieved. For this, take $\mathcal{H}_{L}\left( U_{T}\left(x_{L},t\right)\right)=U\left(x_{L},t\right)-U_{T}\left(x_{L},t\right)$ if the left boundary condition is Dirichlet, and $\mathcal{H}_{L}\left( U_{T}\left( x_{L},t\right)\right)= \mathcal{F}_{L}\left(U_{T}\left(x_{L},t\right),\tfrac{\partial U_{T}\left(x_{L},t\right)}{\partial t}\right)$ if is Neumann, assuming that the solution of the PDE satisfies $\mathcal{F}_{L}\left(U\left(x_{L},t\right),\tfrac{\partial U\left(x_{L},t\right)}{\partial t}\right)=0$. Similarly, define $\mathcal{H}_{R}\left( U_{T}\left( x_{R},t\right)\right)$ considering the right boundary condition. The partial derivatives of $Loss_{LBC}$ and $Loss_{RBC}$ involved in Eq. $(\ref{LossT})$ are given by
\begin{equation*}
\begin{split}
    \frac{\partial Loss_{LBC}}{\partial\omega_{1}^{(j,i)}} & = \frac{2}{\left\vert S_{LBC}\right\vert} \sum_{\left( x_{L},t\right)\in S_{LBC}}  \mathcal{H}_{L}\left( U_{T}\left( x_{L},t\right)\right) \frac{\partial \mathcal{H}_{L}\left( U_{T}\left( x_{L},t\right)\right)}{\partial\omega_{1}^{(j,i)}}, \\
    \frac{\partial Loss_{LBC}}{\partial b_{1}^{(j)}} & = \frac{2}{\left\vert S_{LBC}\right\vert}\sum_{\left( x_{L},t\right)\in S_{LBC}} \mathcal{H}_{L}\left( U_{T}\left( x_{L},t\right)\right) \frac{\partial \mathcal{H}_{L}\left( U_{T}\left( x_{L},t\right)\right)}{\partial b_{1}^{(j)}},\\
    \frac{\partial Loss_{RBC}}{\partial\omega_{1}^{(j,i)}} & = \frac{2}{\left\vert S_{RBC}\right\vert} \sum_{\left( x_{R},t\right)\in S_{RBC}}  \mathcal{H}_{R}\left( U_{T}\left( x_{R},t\right)\right) \frac{\partial \mathcal{H}_{R}\left( U_{T}\left( x_{R},t\right)\right)}{\partial\omega_{1}^{(j,i)}}, \\
    \frac{\partial Loss_{RBC}}{\partial b_{1}^{(j)}} & = \frac{2}{\left\vert S_{RBC}\right\vert}\sum_{\left( x_{R},t\right)\in S_{RBC}} \mathcal{H}_{R}\left( U_{T}\left( x_{R},t\right)\right) \frac{\partial \mathcal{H}_{R}\left( U_{T}\left( x_{R},t\right)\right)}{\partial b_{1}^{(j)}},
\end{split}
\end{equation*}
for $i=1,2$ and $j=1,\ldots ,h_{1}$. Take $\epsilon_{LBC}$ and $\epsilon_{RBC}$ similarly to Eq. $(\ref{E_IC})$ if the corresponding boundary condition is of the Dirichlet type, and $\epsilon_{LBC}^{\prime}$ and $\epsilon_{RBC}^{\prime}$ as in Eq. $(\ref{EP_PDE})$, $\epsilon_{LBC}$ and $\epsilon_{RBC}$ similar to Eq. $(\ref{E_PDE})$  if the corresponding boundary condition is of the Neumann type. Theorem \ref{Coric} implies that there exists $\delta_{\epsilon_{LBC_{j}}}>0$ and $\delta_{\epsilon_{RBC_{j}}}>0$ such that for all
\begin{align*}
&\left(x_{L},t\right)\notin D_{\epsilon_{LBC}}=\left\lbrace \left(x_{L},t\right)\in\left\lbrace x_{L}\right\rbrace\times\Omega_{t} : \left\vert \omega_{1}^{(j,1)}x_{L}+\omega_{1}^{(j,2)}t+ b_{1}^{(j)}\right\vert < \delta_{\epsilon_{LBC_{j}}},j=1,\ldots,h_{1}\right\rbrace ,\\
&\left(x_{R},t\right)\notin D_{\epsilon_{RBC}}=\left\lbrace \left(x_{R},t\right)\in\left\lbrace x_{R}\right\rbrace\times\Omega_{t} : \left\vert \omega_{1}^{(j,1)}x_{R}+\omega_{1}^{(j,2)}t+ b_{1}^{(j)}\right\vert < \delta_{\epsilon_{RBC_{j}}},j=1,\ldots,h_{1}\right\rbrace ,
\end{align*}
it holds that
\begin{equation}
\begin{split}
\label{PLBCtheo2}
&\left\vert\frac{\partial \mathcal{H}_{L}\left( U_{T}\left( x_{L},t\right)\right)}{\partial\omega_{1}^{(j,i)}}\right\vert < \epsilon_{LBC} \ \mathrm{and} \  \left\vert\frac{\partial \mathcal{H}_{L}\left( U_{T}\left( x_{L},t\right)\right)}{\partial b_{1}^{(j)}}\right\vert < \epsilon_{LBC},\\
&\left\vert\frac{\partial \mathcal{H}_{R}\left( U_{T}\left( x_{R},t\right)\right)}{\partial\omega_{1}^{(j,i)}}\right\vert < \epsilon_{RBC} \ \mathrm{and} \  \left\vert\frac{\partial \mathcal{H}_{R}\left( U_{T}\left( x_{R},t\right)\right)}{\partial b_{1}^{(j)}}\right\vert < \epsilon_{RBC},
\end{split}
\end{equation}
for all $\left( x_{R},t\right)\in S_{\epsilon_{RBC}}=S\cap D_{\epsilon_{RBC}}$ for $i=1,2$ and $j=1,\ldots ,h_{1}$, respectively.

Considering $S_{\epsilon}=S_{\epsilon_{IC}}\cup S_{\epsilon_{LBC}}\cup S_{\epsilon_{RBC}}\cup S_{\epsilon_{PDE}}$. From Eq. $(\ref{LossT})$ and the bounds found  in Eqns. $(\ref{PICtheo2})$, $(\ref{PPDEtheo2})$ and $(\ref{PLBCtheo2})$, it follows that
\begin{equation*}
\begin{split}
    \left\vert\frac{\partial Loss_{Total}}{\partial\omega_{1}^{(j,i)}} \right\vert_{S} & = \left\vert\frac{\partial Loss_{IC}}{\partial\omega_{1}^{(j,i)}} +\frac{\partial Loss_{LBC}}{\partial\omega_{1}^{(j,i)}} + \frac{\partial Loss_{RBC}}{\partial\omega_{1}^{(j,i)}} + \frac{\partial Loss_{PDE}}{\partial\omega_{1}^{(j,i)}} \right\vert_{S} ,\\
    & = \left\vert \frac{2}{\left\vert S_{IC}\right\vert} \sum_{\left(x,t_{0}\right)\in S_{IC}} \left( U_{T}\left( x,t_{0}\right)-U_{0}(x)\right) \frac{\partial U_{T}\left(x,t_{0}\right)}{\partial\omega_{1}^{(j,i)}} \right.\\
    & \ \ \ \ \ + \frac{2}{\left\vert S_{LBC}\right\vert} \sum_{\left( x_{L},t\right)\in S_{LBC}}  \mathcal{H}_{L}\left( U_{T}\left( x_{L},t\right)\right) \frac{\partial \mathcal{H}_{L}\left( U_{T}\left( x_{L},t\right)\right)}{\partial\omega_{1}^{(j,i)}} \\
    & \ \ \ \ \ + \frac{2}{\left\vert S_{RBC}\right\vert} \sum_{\left( x_{R},t\right)\in S_{RBC}}  \mathcal{H}_{R}\left( U_{T}\left( x_{R},t\right)\right) \frac{\partial \mathcal{H}_{R}\left( U_{T}\left( x_{R},t\right)\right)}{\partial\omega_{1}^{(j,i)}}\\
    & \ \ \ \ \ \left. + \frac{2}{\left\vert S_{PDE}\right\vert} \sum_{(x,t)\in S_{PDE}}  \mathcal{F}\left( U_{T}(x,t)\right) \frac{\partial \mathcal{F}\left( U_{T}(x,t)\right)}{\partial\omega_{1}^{(j,i)}} \right\vert \\
    & \leq  \left\vert \frac{2}{\left\vert S_{IC}\right\vert} \sum_{\left(x,t_{0}\right)\in S_{IC}-S_{\epsilon_{IC}}} \left( U_{T}\left( x,t_{0}\right)-U_{0}(x)\right) \frac{\partial U_{T}\left(x,t_{0}\right)}{\partial\omega_{1}^{(j,i)}} \right.\\
    & \ \ \ \ \ + \frac{2}{\left\vert S_{LBC}\right\vert} \sum_{\left( x_{L},t\right)\in S_{LBC}-S_{\epsilon_{LBC}}}  \mathcal{H}_{L}\left( U_{T}\left( x_{L},t\right)\right) \frac{\partial \mathcal{H}_{L}\left( U_{T}\left( x_{L},t\right)\right)}{\partial\omega_{1}^{(j,i)}} \\
    & \ \ \ \ \ + \frac{2}{\left\vert S_{RBC}\right\vert} \sum_{\left( x_{R},t\right)\in S_{RBC}-S_{\epsilon_{RBC}}}  \mathcal{H}_{R}\left( U_{T}\left( x_{R},t\right)\right) \frac{\partial \mathcal{H}_{R}\left( U_{T}\left( x_{R},t\right)\right)}{\partial\omega_{1}^{(j,i)}}\\
    & \ \ \ \ \ \left. + \frac{2}{\left\vert S_{PDE}\right\vert} \sum_{(x,t)\in S_{PDE}-S_{\epsilon_{PDE}}}  \mathcal{F}\left( U_{T}(x,t)\right) \frac{\partial \mathcal{F}\left( U_{T}(x,t)\right)}{\partial\omega_{1}^{(j,i)}} \right\vert \\
    & \ \ \ \ \ + \frac{2}{\left\vert S_{IC}\right\vert} \sum_{\left(x,t_{0}\right)\in S_{\epsilon_{IC}}} \left\vert U_{T}\left( x,t_{0}\right)-U_{0}(x)\right\vert \left\vert\frac{\partial U_{T}\left(x,t_{0}\right)}{\partial\omega_{1}^{(j,i)}} \right\vert \\
    & \ \ \ \ \ + \frac{2}{\left\vert S_{LBC}\right\vert} \sum_{\left( x_{L},t\right)\in S_{\epsilon_{LBC}}}  \left\vert \mathcal{H}_{L}\left( U_{T}\left( x_{L},t\right)\right)\right\vert \left\vert \frac{\partial \mathcal{H}_{L}\left( U_{T}\left( x_{L},t\right)\right)}{\partial\omega_{1}^{(j,i)}}\right\vert\\
    & \ \ \ \ \ + \frac{2}{\left\vert S_{RBC}\right\vert} \sum_{\left( x_{R},t\right)\in S_{\epsilon_{RBC}}} \left\vert \mathcal{H}_{R}\left( U_{T}\left( x_{R},t\right)\right) \right\vert \left\vert \frac{\partial \mathcal{H}_{R}\left( U_{T}\left( x_{R},t\right)\right)}{\partial\omega_{1}^{(j,i)}}\right\vert
\end{split}
\end{equation*}
\begin{equation*}
\begin{split}
    \hspace{2.1cm}& \ \ \ \ \ + \frac{2}{\left\vert S_{PDE}\right\vert} \sum_{(x,t)\in S_{\epsilon_{PDE}}} \left\vert \mathcal{F}\left( U_{T}(x,t)\right) \right\vert \left\vert \frac{\partial \mathcal{F}\left( U_{T}(x,t)\right)}{\partial\omega_{1}^{(j,i)}} \right\vert \\
    & < \left\vert \frac{2}{\left\vert S_{IC}\right\vert} \sum_{\left(x,t_{0}\right)\in S_{IC}-S_{\epsilon_{IC}}} \left( U_{T}\left( x,t_{0}\right)-U_{0}(x)\right) \frac{\partial U_{T}\left(x,t_{0}\right)}{\partial\omega_{1}^{(j,i)}} \right.\\
    & \ \ \ \ \ + \frac{2}{\left\vert S_{LBC}\right\vert} \sum_{\left( x_{L},t\right)\in S_{LBC}-S_{\epsilon_{LBC}}}  \mathcal{H}_{L}\left( U_{T}\left( x_{L},t\right)\right) \frac{\partial \mathcal{H}_{L}\left( U_{T}\left( x_{L},t\right)\right)}{\partial\omega_{1}^{(j,i)}} \\
    & \ \ \ \ \ + \frac{2}{\left\vert S_{RBC}\right\vert} \sum_{\left( x_{R},t\right)\in S_{RBC}-S_{\epsilon_{RBC}}}  \mathcal{H}_{R}\left( U_{T}\left( x_{R},t\right)\right) \frac{\partial \mathcal{H}_{R}\left( U_{T}\left( x_{R},t\right)\right)}{\partial\omega_{1}^{(j,i)}}\\
    & \ \ \ \ \ \left. + \frac{2}{\left\vert S_{PDE}\right\vert} \sum_{(x,t)\in S_{PDE}-S_{\epsilon_{PDE}}}  \mathcal{F}\left( U_{T}(x,t)\right) \frac{\partial \mathcal{F}\left( U_{T}(x,t)\right)}{\partial\omega_{1}^{(j,i)}} \right\vert \\
    & \ \ \ \ \ + \frac{2}{\left\vert S_{IC}\right\vert} \sum_{\left(x,t_{0}\right)\in S_{\epsilon_{IC}}} \left\vert U_{T}\left( x,t_{0}\right)-U_{0}(x)\right\vert \epsilon_{IC} + \frac{2}{\left\vert S_{LBC}\right\vert} \sum_{\left( x_{L},t\right)\in S_{\epsilon_{LBC}}}  \left\vert \mathcal{H}_{L}\left( U_{T}\left( x_{L},t\right)\right)\right\vert \epsilon_{LBC} \\
    & \ \ \ \ \ + \frac{2}{\left\vert S_{RBC}\right\vert} \sum_{\left( x_{R},t\right)\in S_{\epsilon_{RBC}}} \left\vert \mathcal{H}_{R}\left( U_{T}\left( x_{R},t\right)\right) \right\vert \epsilon_{RBC}  + \frac{2}{\left\vert S_{PDE}\right\vert} \sum_{(x,t)\in S_{\epsilon_{PDE}}} \left\vert \mathcal{F}\left( U_{T}(x,t)\right) \right\vert \epsilon_{PDE} \\
    & \leq \left\vert \frac{2}{\left\vert S_{IC}\right\vert} \sum_{\left(x,t_{0}\right)\in S_{IC}-S_{\epsilon_{IC}}} \left( U_{T}\left( x,t_{0}\right)-U_{0}(x)\right) \frac{\partial U_{T}\left(x,t_{0}\right)}{\partial\omega_{1}^{(j,i)}} \right.\\
    & \ \ \ \ \ + \frac{2}{\left\vert S_{LBC}\right\vert} \sum_{\left( x_{L},t\right)\in S_{LBC}-S_{\epsilon_{LBC}}}  \mathcal{H}_{L}\left( U_{T}\left( x_{L},t\right)\right) \frac{\partial \mathcal{H}_{L}\left( U_{T}\left( x_{L},t\right)\right)}{\partial\omega_{1}^{(j,i)}} \\
    & \ \ \ \ \ + \frac{2}{\left\vert S_{RBC}\right\vert} \sum_{\left( x_{R},t\right)\in S_{RBC}-S_{\epsilon_{RBC}}}  \mathcal{H}_{R}\left( U_{T}\left( x_{R},t\right)\right) \frac{\partial \mathcal{H}_{R}\left( U_{T}\left( x_{R},t\right)\right)}{\partial\omega_{1}^{(j,i)}}\\
    & \ \ \ \ \ \left. + \frac{2}{\left\vert S_{PDE}\right\vert} \sum_{(x,t)\in S_{PDE}-S_{\epsilon_{PDE}}}  \mathcal{F}\left( U_{T}(x,t)\right) \frac{\partial \mathcal{F}\left( U_{T}(x,t)\right)}{\partial\omega_{1}^{(j,i)}} \right\vert + \frac{\epsilon}{4} + \frac{\epsilon}{4} + \frac{\epsilon}{4} + \frac{\epsilon}{4}\\
    & = \left\vert \frac{2}{\left\vert S_{IC}\right\vert} \sum_{\left(x,t_{0}\right)\in S_{IC}-S_{\epsilon_{IC}}} \left( U_{T}\left( x,t_{0}\right)-U_{0}(x)\right) \frac{\partial U_{T}\left(x,t_{0}\right)}{\partial\omega_{1}^{(j,i)}} \right.\\
    & \ \ \ \ \ + \frac{2}{\left\vert S_{LBC}\right\vert} \sum_{\left( x_{L},t\right)\in S_{LBC}-S_{\epsilon_{LBC}}}  H_{L}\left( U_{T}\left( x_{L},t\right)\right) \frac{\partial H_{L}\left( U_{T}\left( x_{L},t\right)\right)}{\partial\omega_{1}^{(j,i)}} \\
    & \ \ \ \ \ + \frac{2}{\left\vert S_{RBC}\right\vert} \sum_{\left( x_{R},t\right)\in S_{RBC}-S_{\epsilon_{RBC}}}  H_{R}\left( U_{T}\left( x_{R},t\right)\right) \frac{\partial H_{R}\left( U_{T}\left( x_{R},t\right)\right)}{\partial\omega_{1}^{(j,i)}}\\
    & \ \ \ \ \ \left. + \frac{2}{\left\vert S_{PDE}\right\vert} \sum_{(x,t)\in S_{PDE}-S_{\epsilon_{PDE}}}  F\left( U_{T}(x,t)\right) \frac{\partial F\left( U_{T}(x,t)\right)}{\partial\omega_{1}^{(j,i)}} \right\vert + \epsilon .\\
\end{split}
\end{equation*}
Therefore, if $\Omega_{x}\times\Omega_{t}$ is relatively large, it is possible to take $\epsilon$ small enough, such that

\begin{equation*}
\begin{split}
    \left\vert\frac{\partial Loss_{Total}}{\partial\omega_{1}^{(j,i)}} \right\vert_{S} & < \left\vert \frac{2}{\left\vert S_{IC}-S_{\epsilon_{IC}}\right\vert} \sum_{\left(x,t_{0}\right)\in S_{IC}-S_{\epsilon_{IC}}} \left( U_{T}\left( x,t_{0}\right)-U_{0}(x)\right) \frac{\partial U_{T}\left(x,t_{0}\right)}{\partial\omega_{1}^{(j,i)}} \right.\\
    & \ \ \ \ \ + \frac{2}{\left\vert S_{LBC}-S_{\epsilon_{LBC}}\right\vert} \sum_{\left( x_{L},t\right)\in S_{LBC}-S_{\epsilon_{LBC}}}  \mathcal{H}_{L}\left( U_{T}\left( x_{L},t\right)\right) \frac{\partial \mathcal{H}_{L}\left( U_{T}\left( x_{L},t\right)\right)}{\partial\omega_{1}^{(j,i)}} \\
    & \ \ \ \ \ + \frac{2}{\left\vert S_{RBC}-S_{\epsilon_{RBC}}\right\vert} \sum_{\left( x_{R},t\right)\in S_{RBC}-S_{\epsilon_{RBC}}}  \mathcal{H}_{R}\left( U_{T}\left( x_{R},t\right)\right) \frac{\partial \mathcal{H}_{R}\left( U_{T}\left( x_{R},t\right)\right)}{\partial\omega_{1}^{(j,i)}}\\
    & \ \ \ \ \ \left. + \frac{2}{\left\vert S_{PDE}-S_{\epsilon_{PDE}}\right\vert} \sum_{(x,t)\in S_{PDE}-S_{\epsilon_{PDE}}}  \mathcal{F}\left( U_{T}(x,t)\right) \frac{\partial \mathcal{F}\left( U_{T}(x,t)\right)}{\partial\omega_{1}^{(j,i)}} \right\vert \\
    & = \left\vert\frac{\partial Loss_{IC}}{\partial\omega_{1}^{(j,i)}} +\frac{\partial Loss_{LBC}}{\partial\omega_{1}^{(j,i)}} + \frac{\partial Loss_{RBC}}{\partial\omega_{1}^{(j,i)}} + \frac{\partial Loss_{PDE}}{\partial\omega_{1}^{(j,i)}} \right\vert_{S-S_{\epsilon}}\\
    & = \left\vert\frac{\partial Loss_{Total}}{\partial\omega_{1}^{(j,i)}} \right\vert_{S-S_{\epsilon}},
\end{split}
\end{equation*}    
similarly,
\begin{equation*}
\begin{split}
    \left\vert\frac{\partial Loss_{Total}}{\partial b_{1}^{(j)}} \right\vert_{S} & < \left\vert\frac{\partial Loss_{Total}}{\partial b_{1}^{(j)}} \right\vert_{S-S_{\epsilon}},
\end{split}
\end{equation*}
for $i=1,2$ and $j=1,\ldots ,h_{1}$.

\end{proof}

Note that the proof of Theorem \ref{theopde} is based on bounding the gradient of the loss function. This is done by bounding the partial all mixed derivatives of the trial solution $U_{T}(x,t)$ that arise from deriving the initial and boundary conditions, and the PDE with respect to $\omega_{1}^{(j,i)}$ and $b_{1}^{(j)}$, for $i=1,2$ and $j=1,\ldots ,h_{1}$. In particular, Theorem \ref{Coric} allows us to find the necessary bounds for all those mixed derivatives. Since Corollary \ref{Cortheo1} allows us to extend the result presented in Theorem \ref{Coric}, adding the requirement that all terms in the PDE have as a factor at least one partial derivative of $U_{T}(x,t)$, adding this requirement to Theorem \ref{theopde} allows us to use Corollary \ref{Cortheo1} instead of Theorem \ref{Coric} to prove the following corollary.

\begin{corollary}
\label{corpde}
Given $\mathcal{F}\left( U\right)=0$ a PDE with initial and boundary conditions over $\overline{\Omega_{x}\times\Omega_{t}}$, such that it only contains terms with at least one factor equal to a partial derivative of $U$. Let $\epsilon \ll 1$, $S$ a sample of points within the domain of $\mathcal{F}$, and $U_{T}(x,t)$ given by Eq. $(\ref{NNs})$. There are $\delta_{\epsilon_{j}}>0$ such that
\begin{equation}
\label{Theo2}
\left\vert\frac{\partial Loss_{Total}}{\partial\omega_{k}^{(j,i)}} \right\vert_{S}<\left\vert\frac{\partial Loss_{Total}}{\partial\omega_{k}^{(j,i)}} \right\vert_{S-S_{\epsilon}} \ \ \ and \ \ \ \left\vert\frac{\partial Loss_{Total}}{\partial b_{k}^{(j)}} \right\vert_{S}<\left\vert\frac{\partial Loss_{Total}}{\partial b_{k}^{(j)}} \right\vert_{S-S_{\epsilon}}
\end{equation}
for $i=1,2$ $j=1,\ldots ,h_{k}$ and $k=2,\ldots ,F+1$, where $S_{\epsilon}=S_{\epsilon_{IC}}\cup S_{\epsilon_{LBC}}\cup S_{\epsilon_{RBC}}\cup S_{\epsilon_{PDE}}$ is a set of the form
\begin{equation}
\label{Dgzcor3}
S_{\epsilon}= \left\lbrace (x,t)\in S^{\prime}\subset \overline{\Omega_{x}\times\Omega_{t}} : \left\vert \omega_{1}^{(j,1)}x+\omega_{1}^{(j,2)}t+ b_{1}^{(j)}\right\vert >\delta^{\prime}_{\epsilon_{j}} \ \forall j\leq h_{1}\right\rbrace ,
\end{equation}
with $S^{\prime}\in\left\lbrace S_{IC},S_{LBC},S_{RBC},S_{PDE}\right\rbrace$.
\end{corollary}

\begin{remark}
For the development of the sampling method that allows improving the optimization process of $\omega_{1}^{(j,i)}$ and $b_{1}^{(j)}$, we refer to $S_{\epsilon}$, given by Eq. $(\ref{Dgztheo2})$ on Theorem \ref{theopde}, as the diminishing gradient zone of $U_{T}(x,t)$. This because for all $(x,t)\in S_{\epsilon}$ it holds that $\frac{\partial^{m+1} U_{T}(x,t)}{\partial\omega_{1}^{(j,i)}\partial x^{p}\partial t^{m-p}}$ and $\frac{\partial^{m+1} U_{T}(x,t)}{\partial\omega_{1}^{(j,i)}\partial x^{p}\partial t^{m-p}}$ are negligible for $i=1,2$, $j=1,\ldots ,h_{1}$ and $0\leq p\leq m\leq n$, meaning that $S_{\epsilon}$ is equivalent to a diminishing gradient zone of $U_{T}(x,t)$ that takes into account only the parameters $\omega_{1}^{(j,i)}$ and $b_{1}^{(j)}$. Similarly, when applicable, we refer to $S_{\epsilon}$, given by Eq. $(\ref{Dgzcor3})$ on Corollary \ref{corpde}, as the diminishing gradient zone of $U_{T}(x,t)$.
\end{remark}

When considering a trial solution $U_{T}(x,t)$ with $F$ hidden layers and $h_{k}$ neurons in the $k$-th hidden layer for a PDE, Theorem \ref{theopde} implies that the step size with which the parameters $\omega_{1}^{(j,i)}$ and $b_{1}^{(j)}$ are updated decreases when considering points in the sample contained in the diminishing gradient zone of $U_{T}(x,t)$. That is, using points in the domain of the equation which cancel the derivatives with respect to $\omega_{1}^{(j,i)}$ and $b_{1}^{(j)}$ is far from being ideal for its optimization. Since considering samples completely contained in the active gradient zone of $U_{T}(x,t)$ would help prevent the gradient of the loss function from being negligible, this would help mitigate the vanishing gradient problem, which affects the parameters of the first hidden layer to a greater extent. Even more, this result extends to all parameters $\omega_{k}^{(j,i)}$ and $b_{k}^{(j)}$ in $U_{T}(x,t)$ for all PDEs that meet the requirements of Corollary $\ref{corpde}$, for example, the advection equation shown in Section 5. With this in mind, we will proceed to create a method that allows us to select samples within regions that approximate the active gradient zone of $U_{T}(x,t)$.


\section{Stratified Sampling Algorithms}

Generating a proper set of points for the test function is crucial for network parameter optimization. One straightforward method is to adopt the random sampling scheme (i.e., the test points are selected in the PDE domain randomly and uniformly). However, the uniform sampling scheme is inefficient if the expected solutions to the given PDE have local nonsmooth structures, such as the two-scale solutions considered in this paper. For this case, nonlinear sampling schemes that select test points nonuniformly are necessary. For example, the SS scheme is a standard technique that reduces sampling errors by partitioning the domain into multiple subdomains, known as strata \cite{liberty2016stratified}. The number of subdomains is determined by the local regularity of the solution for this case.

As mentioned, the first stage of the optimization process in this paper consists of training the network to satisfy the initial condition with $Loss_{IC}$ as the loss function. For this stage, from Theorem \ref{theopde}, we can create subdomains from the given domain $[x_{L},x_{R}]\times\lbrace t_{0}\rbrace$ by considering the active gradient zone of $Loss_{IC}$ with respect to the parameters of the first network layer. Once these subdomains are determined, the training of the network parameters is conducted by selecting random points uniformly within it. The proposed algorithm for network training, considering only the initial condition, is explained in Algorithm 1. \newline

\noindent\textbf{Algorithm 1:} Sample selection algorithm used to train the initial condition.
\begin{enumerate}[I.]
    \item For each $j\in\lbrace 1,\ldots ,h_{1}\rbrace$, find an approximation to the active gradient zone of $N_{1}^{(j)}\left( x,t_{0}\right)$ within the domain of the equation, given by $\left[ x_{L}^{(j)},x_{R}^{(j)}\right]$. To do this, it is necessary to compute:
    \begin{itemize}[-]
        \item Center of the active gradient zone, denoted by $C^{(j)}$, defined as the value of $x$ for which $N_{1}^{(j)}\left( x,t_{0}\right)=\frac{1}{2}$ and $C^{(j)}=-\frac{t_{0}\omega_{1}^{(j,2)}+b_{1}^{(1)}}{\omega_{1}^{(j,1)}}$. 
        \item Active gradient zone radius, denoted by $\delta^{(j)}=\frac{\epsilon}{\omega_{1}^{(j,1)}}$, which satisfy the condition that $N_{1}^{(j)}\left( C^{(j)}+\delta^{(j)},t_{0}\right)<\epsilon$ or $N_{1}^{(j)}\left( C^{(j)}+\delta^{(j)},t_{0}\right)>1-\epsilon$.
        \item Intersection of the active gradient zones of $N_{1}^{(j)}\left( x,t_{0}\right)$ and the domain $\left[ x_{L},x_{R}\right]$, given by $\left[ x_{L}^{(j)},x_{R}^{(j)}\right] = \left[\max{\left\lbrace x_{L}, C^{(j)}-\delta^{(j)}\right\rbrace},\min{\left\lbrace x_{R}, C^{(j)}+\delta^{(j)}\right\rbrace}\right]$.
    \end{itemize}
    \item Take a random sample of points $S_{IC}$  in the region that results from the union of all $\left[ x_{L}^{(j)},x_{R}^{(j)}\right]$. To create this sample, it is necessary to perform the following steps: 
    \begin{itemize}[-]
        \item Sort all subdomains $\left[ x_{L}^{(j)},x_{R}^{(j)}\right]$ and check if there are intersections among them. If there are any intersections, take the union of the intervals instead of each one of them.
        \item Take a random sample $S^{\left( j^{\prime}\right)}$ within the intervals obtained in the previous step. The number of points in each $S^{\left( j^{\prime}\right)}$ is determined based on the length of the interval and the desired density of points in the sample over the entire domain denoted by $d_{p}$.
        \item Define $S_{IC}$ as the union of all $S^{\left( j^{\prime}\right)}$.
    \end{itemize}
\end{enumerate}

The pseudocode of Algorithm 1 is provided in the appendix. The proposed algorithm is based on the fact that the first network layer comprises $h_{1}$ sigmoid functions evaluated with $\omega_{1}^{(j,1)}x+\omega_{1}^{(j,2)}t+b_{1}^{(j)}$ with $j=1,\ldots ,h_{1}$, which makes it possible to approximate their active and diminishing gradient zones. Thus, the union of the active gradient zone of all these sigmoid functions serves as the region where the gradient of $Loss_{IC}$ for the weights of the first network layer is nonnegligible.

Figure $\ref{fig:ex1v}$ shows an example of applying the previous algorithm to approximate the stiff initial condition $U(x)$ given by
\begin{equation}
    \label{function_example1}
    U(x)=\left\lbrace\begin{split}
        1 \ \ \ \ \ \ & \ \ \ \mathrm{if} \ x\leq -3\pi\\
        \sin\left(\frac{x}{2}\right) & \ \ \ \mathrm{if} \ -3\pi \leq x\leq 3\pi\\
        -1 \ \ \ \ \ & \ \ \ \mathrm{if} \ \pi\leq x.
    \end{split}\right.
\end{equation}
In the figure, the dotted red line represents the exact given initial condition $u(x,t_0)$, the solid blue line represents the approximated trial function $U_T\left(x,t_0\right)$, and the blue dots represents the sample points selected through Algorithm 1. 
The approximation, $U_{T}(x)$, is obtained by training over the sample sets of points randomly selected within the active gradient zones given by Eq. $(\ref{NNs})$. As indicated in the figure, the selected sample points are concentrated near the highly localized areas of the given $U(x,t_0)$. No sample points were selected within the plateau of $U(x,t_0)$.

\begin{figure}[h]
\centering
\begin{subfigure}{1\textwidth}
  \centering
  \includegraphics[scale=.65]{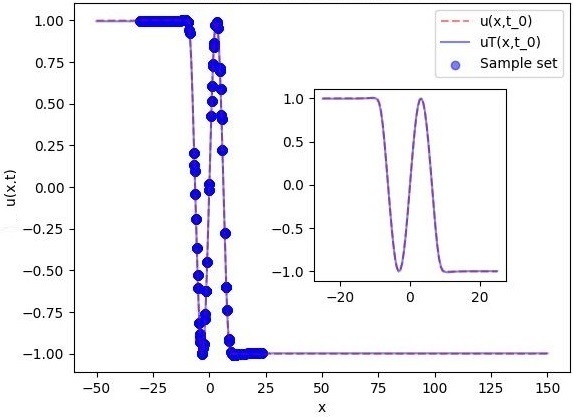}
\end{subfigure}
\caption{Exact initial condition $U(x, t_0)$ (dotted red line) given by Eq. $(\ref{function_example1})$ over the domain $\Omega_{x}=[-50,150]$ and its approximation $U_{T}(x)$ (solid blue line) determined by Eq. $(\ref{NNs})$ with three hidden layers and $60$ neurons per layer. The blue dots represent the sample points in the active gradient zones over $100,000$ epochs.}
\label{fig:ex1v}
\end{figure}

Figure \ref{fig:ex1} depicts the convergent behavior of the loss function with the training process, and Fig. \ref{fig:ex2} presents the size of the sample sets (total points in the sample sets) for each epoch. We used the constant density of points for the area where sample points were selected. For the uniform classical sampling (CS) method, sample points were selected in the entire domain with the density $d_{p}$, and for the SS method, the samples points were selected in the active gradient zones with the same density $d_{p}$ (see Remark 1). The figure indicates that the loss function decreases with increasing epochs, and the total points also decrease, although slowly. The selected sample points over 100,000 epochs are listed in Fig. \ref{fig:ex1v}. The chosen sample points define the highly localized regions of the initial condition well. 

\begin{figure}[h]
\centering
\begin{subfigure}{.5\textwidth}
  \centering
  \includegraphics[scale=.56]{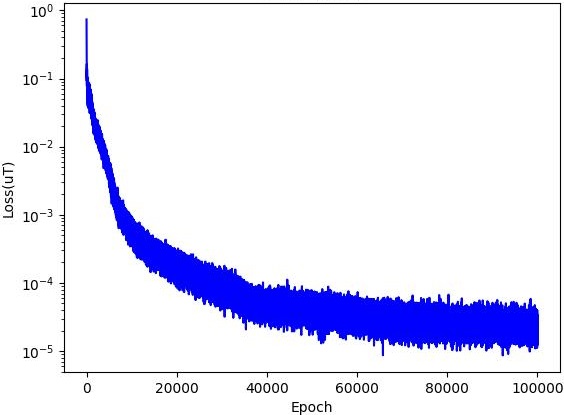}
  \caption{Loss function value per epoch.}
  \label{fig:ex1}
\end{subfigure}%
\begin{subfigure}{.5\textwidth}
  \centering
  \includegraphics[scale=.56]{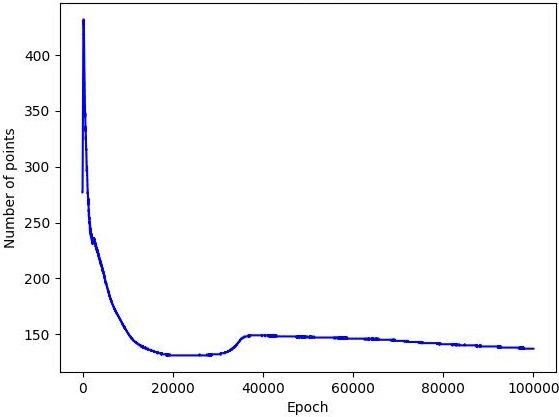}
  \caption{Sample set size per epoch.}
  \label{fig:ex2}
\end{subfigure}
\caption{(a) Convergence behavior of the loss function over the epochs. (b) Size of the sample sets for $U(x, t_0)$ given by Eq. $(\ref{function_example1})$ over the domain $\Omega_{x}=[-50,150]$ using the proposed sampling algorithm over $100,000$ epochs.}
\label{fig:ex2v}
\end{figure}

\begin{remark}
Algorithm 1 maintains a specific point density per unit volume in the selected domain to ensure a fair comparison with the usual learning methods that apply the uniform sampling scheme. Therefore, the density can be adjusted adaptively depending on the regularity of the function to enhance efficiency.

\end{remark}

Now, for the general optimization, for which $Loss_{Total}$ (Eq. (\ref{Loss})) is taken as the loss function, it is also possible to separate the domain of the equation into strata. This will depend on the gradient of the $Loss_{PDE}$  with respect to the weights of the first layer of the network. For this, let us note that $\sigma\left(\omega_{1}^{(j,1)}x+\omega_{1}^{(j,2)}t+b_{1}^{(j)}\right) = \sigma\left(\omega_{1}^{(j,1)}\left( x+{{\omega_{1}^{(j,2)}}\over{\omega_{1}^{(j,1)}}}t\right)+b_{1}^{(j)}\right)$ is a traveling wave with constant speed of $-{{\omega_{1}^{(j,2)}}\over{\omega_{1}^{(j,1)}}}$ for every $j=1,\ldots, h_{1}$ with which it is easy to find the union of their active gradient zones. However, generating a sample of points in this region is not straightforward in terms of computation. To address this issue, we note that it is possible to split the domain of the equation into subdomains of the form $(x_{L},x_{R})\times(t_{i},t_{i+1}]$ for $i=0,\ldots, N-1$, that is, the time domain is also split. With this split, it is possible to delimit the region that completely contains the intersections of $(x_{L},x_{R})\times(t_{i},t_{i+1}]$ with the union of the active gradient zones of all neurons in the first hidden layer. Therefore, the union of these regions within each subdomain $(x_{L},x_{R})\times(t_{i},t_{i+1}]$ yields the approximate zone of the active gradient zone for the loss function of $Loss_{PDE}$ with respect to the weights of the first layer of the network. In this way, the sample selected by our method must be contained in this approximate active gradient zone. The following is the proposed algorithm to generate the random samples over the interior domain, Algorithm 2:\newline

\noindent\textbf{Algorithm 2:} Sample selection algorithm used to train the PDE.
\begin{enumerate}[I.]
    \item For each $i=0,\ldots ,N-1$, define $t_{i}=t_{0}+i\Delta_{t}$, where $\Delta_{t} = \dfrac{(t_{F}-t_{0})}{N}$ with the final time $t_F$.
    \item For each $i=0,\ldots ,N-1$ and $j\in\lbrace 1,\ldots ,h_{1}\rbrace$, find an approximation to the active gradient zone of $N_{1}^{(j)}\left( x,t\right)$ within the subdomain of the equation with $t\in\left( t_{i},t_{i+1}\right]$, given by $\left[ x_{L,i}^{(j)},x_{R,i}^{(j)}\right]\times\left( t_{i},t_{i+1}\right]$. To do this, it is necessary to compute the followings:
    \begin{itemize}[-]
        \item Centers of the active gradient zones of $N_{1}^{(j)}\left( x,t\right)$ at $t_{i}$ and $t_{i+1}$, denoted by $C_{i}^{(j)}$ and $C_{i+1}^{(j)}$, respectively, given by $C_{i}^{(j)}=-\frac{t_{i}\omega_{1}^{(j,2)}+b_{1}^{(1)}}{\omega_{1}^{(j,1)}}$ and $C_{i+1}^{(j)}=-\frac{t_{i+1}\omega_{1}^{(j,2)}+b_{1}^{(1)}}{\omega_{1}^{(j,1)}}$. 
        \item Active gradient zone radius, denoted by $\delta^{(j)}=\frac{\epsilon}{\omega_{1}^{(j,1)}}$, which satisfy the conditions that $N_{1}^{(j)}\left( C^{(j)}+\delta^{(j)},t_{0}\right)<\epsilon$ or $N_{1}^{(j)}\left( C^{(j)}+\delta^{(j)},t_{0}\right)>1-\epsilon$.
        \item Intersection of the active gradient zone of $N_{1}^{(j)}\left( x,t\right)$ and the domain $\left[ x_{L},x_{R}\right]\times\left( t_{i},t_{i+1}\right]$ given by $\left[ x_{L,i}^{(j)},x_{R,i}^{(j)}\right] = \left[\max{\left\lbrace x_{L}, \min{ \left\lbrace C_{i}^{(j)}-\delta^{(j)}, C_{i+1}^{(j)}-\delta^{(j)} \right\rbrace}\right\rbrace},\min{\left\lbrace x_{R}, \max{ \left\lbrace C_{i}^{(j)}+\delta^{(j)}, C_{i+1}^{(j)}+\delta^{(j)} \right\rbrace}\right\rbrace}\right]$.
    \end{itemize}
    \item Take a random sample of points $S_{PDE}$  in the region that results from the union of all $\left[ x_{L,i}^{(j)},x_{R,i}^{(j)}\right]\times\left( t_{i},t_{i+1}\right]$. To create this sample, it is necessary to perform the following steps: 
    \begin{itemize}[-]
        \item For each $i=0,\ldots ,N-1$ sort all $\left[ x_{L,i}^{(j)},x_{R,i}^{(j)}\right]$ and check if there are intersections among them. If there are any intersections, consider the union of the intervals instead of each of them individually.
        \item Take a random sample $S_{i}^{\left( j^{\prime}\right)}$ within the subdomains $\left[ x_{L,i}^{\left( j^{\prime}\right)},x_{R,i}^{\left( j^{\prime}\right)}\right]\times\left( t_{i},t_{i+1}\right]$. The number of points in each $S^{\left( j^{\prime}\right)}$ is determined based on the length of the interval and the desired density of points in the sample over the entire domain of the equation, denoted by $d_{p}$.
        \item Define $S_{PDE}$ as the union of all $S_{i}^{\left( j^{\prime}\right)}$.
    \end{itemize}
\end{enumerate}

\ref{AppendixA} provides the pseudocode. Thus, for optimization in the general case, Algorithms 1 and 2 for sampling must be used to compute $Loss_{IC}$ and $Loss_{PDE}$, respectively, in addition to the necessary samples for $Loss_{BC}$. Figure \ref{fig5} presents a graphical representation of the sampling regions at a particular epoch found by the algorithms using the active and diminishing gradient zones of the loss function. In the figure, the time domain of $[t_0, t_F]$ is split into $[t_0, t_1], \cdots, [t_{N-1}, t_N]$, and the selected sample domains are represented in blue boxes in each time interval. In each sample domain, the active gradient zone is represented by a solid red line, with the center represented by a dotted red line. In contrast to the active gradient zone, its complement, the diminishing gradient zone, is represented by gray areas. The propagation of the active gradient zone is based on the traveling wave speed defined in the first layer. At certain time intervals, multiple active gradient zones overlap.

\begin{figure}[h]
	\centering
	\includegraphics[scale=.65]{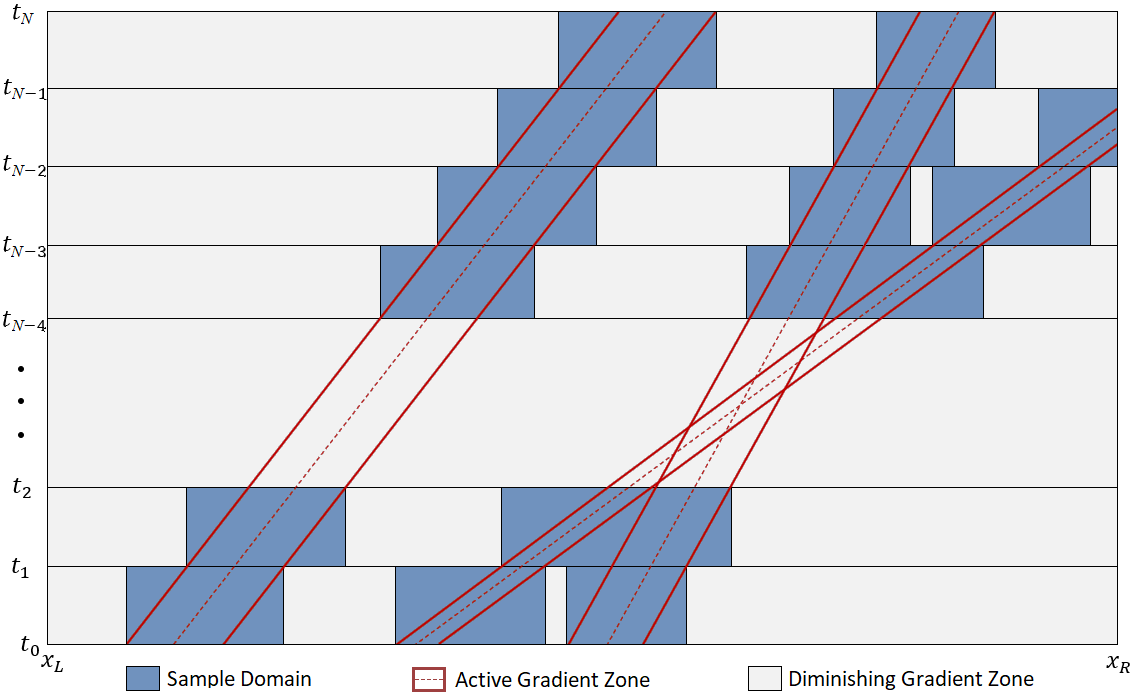}
	\caption{Active and diminishing gradient zone schematic. The time domain is split into $N$ subintervals. The sample domain with the active gradient zone (solid red line centered between dotted red lines) is marked by blue boxes. Gray areas mark diminishing gradient zones. Sample points in blue subdomains train $U_{T}(x,t)$ given by Eq. (\ref{NNs}). Color available online. }
	\label{fig5}
\end{figure}


\section{Numerical Results}

We focused on the PDEs with traveling wave solutions with highly localized two-scale solution profiles to validate the proposed algorithms. We compared the numerical results produced by the proposed and traditional methods. For the numerical experiments, we scaled the test PDEs to train the network for the solutions defined on a large domain and to determine how the stiff profiles of the solutions are preserved for a large domain. Consequently, we expected that the absolute values of the network weights for the scaled PDEs would become smaller than those of the network for the original PDEs.


\subsection{Advection Equation}

We considered the advection equation \cite{vadyala2022physics}, which has the traveling wave solution with a constant speed of $c$ and a highly localized initial condition, given by the following:
\begin{equation}
\label{Acvec}
u_{t}+cu_{x}=0.
\end{equation}
If the initial condition is given as $u(x,t_{0})=u_{0}(x)$, the exact solution of this equation is given by
\begin{equation}
\label{soladv}
u(x,t)=u_{0}(x-ct). 
\end{equation}
To demonstrate the performance of the proposed method, we considered the case where the initial condition exhibits a sharp gradient in a small, localized region but is otherwise highly smooth, maintaining a long-range plateau. Additionally, we selected a large domain where the solution is defined. Given that the region of the stiff solution is highly localized compared to the entire domain, approximating such a solution using the neural network approach with a uniform sampling scheme becomes computationally inefficient and expensive.

Table \ref{tabAD} lists the minimum loss function value for the advection equation with $c = 1$ for various domain sizes and epoch numbers. The initial condition was given as $u_0(x) = e^{-x^2}$. The physical boundary condition is defined as $u(x_L,t) = 0$. To ensure that the solution remains flat near the right boundary until the final time, we also enforced $u(x_R,t) = 0$. In the table, we compared the proposed SS algorithm with the uniform CS algorithm. The CS loss function value in the table represents the smaller result for the loss function obtained using the CS scheme through the training process, whereas the SS loss function value represents the smaller result of the proposed sampling scheme. Three hidden layers are used for the network, each with 20 neurons.
\begin{center}
\begin{tabular}{c c c c}
\hline
Domain & Epoch & CS Loss function value & SS Loss function value\\
\hline
$\left[ -20,80\right]\times \left[ 0,60\right]$ & 20,000 & $1.89\times 10^{-7}$ & $2.97\times 10^{-7}$ \\
$\left[ -200,800\right]\times \left[ 0,600\right]$ & 50,000 & $5.24\times 10^{-7}$ & $8.79\times 10^{-8}$ \\
$\left[ -2000,8000\right]\times \left[ 0,6000\right]$ & 100,000 & $3.19\times 10^{-6}$ & $9.86\times 10^{-7}$ \\
\hline
\end{tabular}
\captionof{table}{Minimum loss function value for the advection equation (Eq. $(\ref{Acvec})$) with $c=1$ for various domain sizes $\left[ x_{L},x_{R}\right]\times \left[ t_{0},t_{F}\right]$ for the initial condition $u_{0}(x)=e^{-x^{2}}$ and the boundary condition $u\left( x_{L},t\right) = 0$. We also enforce $u\left( x_{R},t\right)=0$. The network for all trial solutions consists of three hidden layers with 20 neurons for each.}
\label{tabAD}
\end{center}

Table \ref{tabAD} presents a decrease in the loss function value using SS techniques compared to using CS methods. Although a lower loss function value generally does not imply a better approximation of the solution to a PDE, as indicated in Fig. \ref{fig:AD}, using SS techniques improves optimization.

Using SS does not always represent a decrease in the loss function value, as listed in Tables \ref{tabFish} and \ref{tabZel} for the Fisher and Zeldovich equations, respectively. This result is partly due to the domain sizes considered for this equation, but more importantly, because the advection equation satisfies the requirements of Corollary \ref{corpde}. This outcome implies that using SS represents an improvement in optimizing all parameters in $U_{T}(x,t)$, not just those in the first hidden layer.

Figure \ref{fig:AD} illustrates the solutions in the $x$-$t$ domain, the solution using the CS method (left), SS method (middle), and exact solution (right), given by $u(x,t)=e^{-(x-t)^{2}}$. The figure reveals that the SS method offers a better approximation, although both the CS and SS methods have similar loss function values. The amplitude of the solution obtained using the CS method is visibly smaller than the amplitude of the exact solution, whereas the amplitude of the solution with the SS method is closer to the exact amplitude. Based on the experiments, the advantage of the proposed method becomes more evident as the domain size and number of neurons per layer increase.

\begin{figure}[h]
\centering
\begin{subfigure}{.33\textwidth}
  \centering
  \includegraphics[scale=.55]{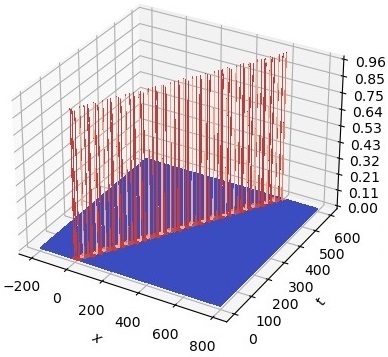}
  \caption{Trial solution using CS.}
  \label{fig:AD1}
\end{subfigure}%
\begin{subfigure}{.33\textwidth}
  \centering
  \includegraphics[scale=.55]{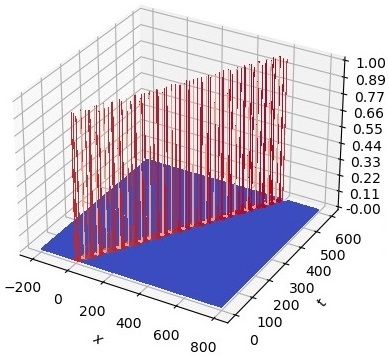}
  \caption{Trial solution using SS.}
  \label{fig:AD2}
\end{subfigure}%
\begin{subfigure}{.33\textwidth}
  \centering
  \includegraphics[scale=.55]{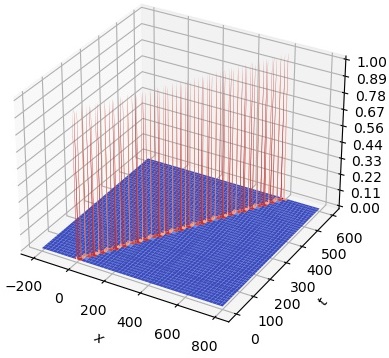}
  \caption{Exact solution.}
  \label{fig:AD3}
\end{subfigure}
\caption{Trial solutions obtained after 50,000 epochs using the (a) classical sampling (CS) and (b) stratified sampling (SS) algorithms in order to approximate (c) the exact solution to the advection equation (Eq. $(\ref{Acvec})$) over the domain $\Omega_{x}\times\Omega_{t}=[-200,800]\times [0,600]$.}
\label{fig:AD}
\end{figure}

Figure \ref{fig:AD_IC} illustrates the initial condition and its approximation with both the CS (Fig. \ref{fig:A1IC}) and SS (Fig. \ref{fig:A2IC}) methods. The dotted red line represents the initial condition, and the solid blue line approximates it. The figure reveals that the proposed method yields a better approximation of the initial condition, which is clear because the maximum value of $U_{T}\left(x,t_{0}\right)$, given by the CS method, is approximately 0.96 instead of 1, as it should be. Moreover, in Fig. \ref{fig:AD_IC}, the value of $U_{T}\left(x_{R},t_{0}\right)$ is closer to the actual value of $U\left(x_{R},t_{0}\right)$ when using SS, even when $\left(x_{R},t_{0}\right)$ was never considered in the sample.

\begin{figure}[h]
\centering
\begin{subfigure}{.5\textwidth}
  \centering
  \includegraphics[scale=.56]{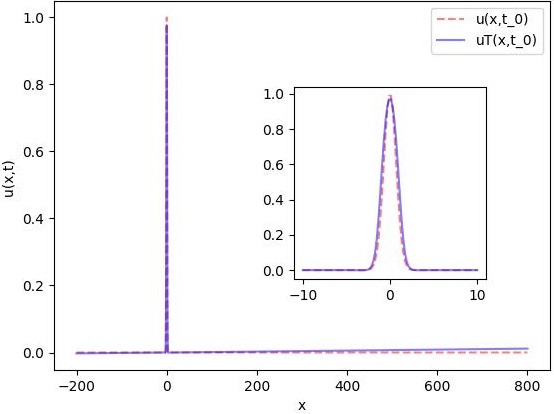}
  \caption{Initial condition using CS.}
  \label{fig:A1IC}
\end{subfigure}%
\begin{subfigure}{.5\textwidth}
  \centering
  \includegraphics[scale=.56]{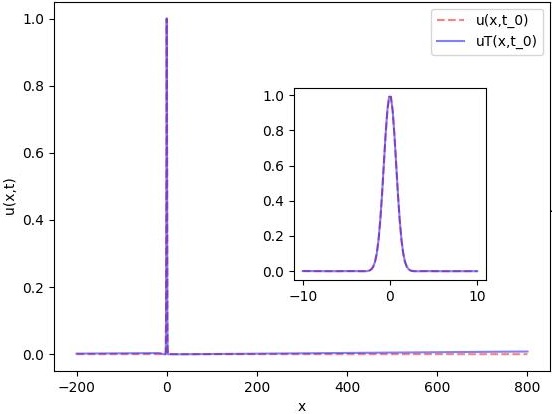}
  \caption{Initial condition using SS.}
  \label{fig:A2IC}
\end{subfigure}
\caption{Approximations of the initial condition  over the domain $\Omega_{x}\times\Omega_{t}=[-200,800]\times [0,600]$ with (a) the classical sampling method and (b) the proposed stratified sampling method after 50,000 epochs.}
\label{fig:AD_IC}
\end{figure}

Figure \ref{fig:AD_Loss} illustrates how the loss function changes as the training progresses with the CS (Fig. \ref{fig:A1}) and SS 
(Fig. \ref{fig:A2}) methods. The loss function seems saturated when the epoch number is high with the CS method, whereas the loss function seems to decay further, although it highly fluctuates with the SS method. Figure \ref{fig:A3} depicts the corresponding MSE with respect to $u(x,t)=e^{-(x-t)^{2}}$ for the CS method (dotted red line) and the SS method (solid blue line). In the figure, the proposed SS method yields better convergence behavior than the CS method. Figure \ref{fig:A4} displays how many sampling points are selected using SS as the training progresses, maintaining the same point density as the CS method, which used 15,000 points.\\

\begin{figure}[h]
\centering
\begin{subfigure}{.5\textwidth}
  \centering
  \includegraphics[scale=.55]{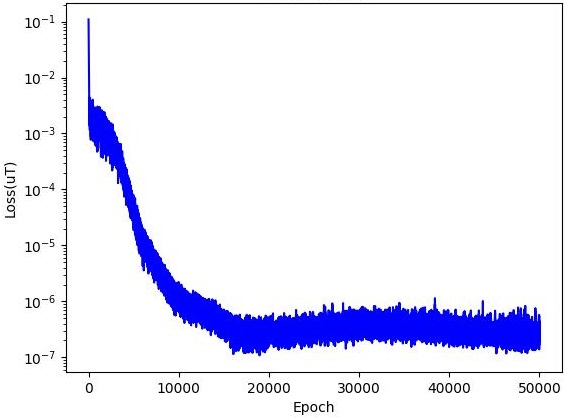}
  \caption{Loss function using CS.}
  \label{fig:A1}
\end{subfigure}%
\begin{subfigure}{.5\textwidth}
  \centering
  \includegraphics[scale=.55]{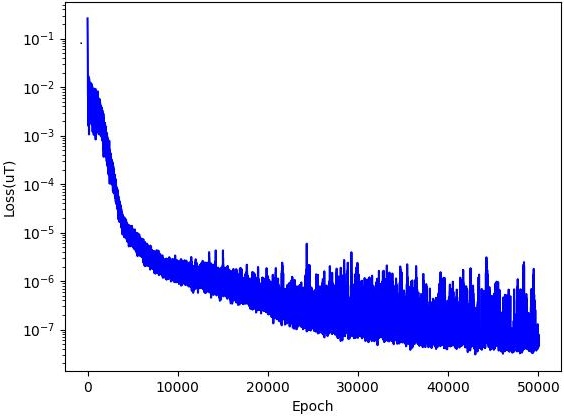}
  \caption{Loss function using SS.}
  \label{fig:A2}
\end{subfigure}
\begin{subfigure}{.5\textwidth}
  \centering
  \includegraphics[scale=.55]{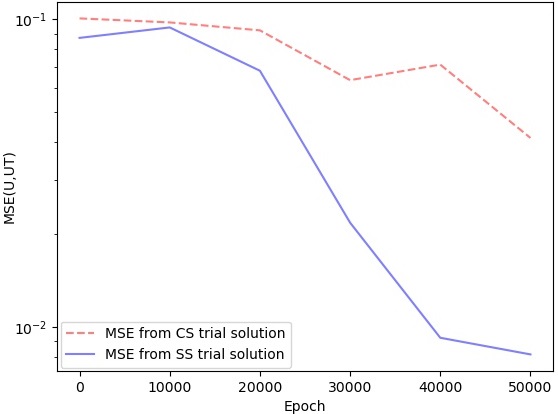}
  \caption{MSE respect to the exact solution.}
  \label{fig:A3}
\end{subfigure}%
\begin{subfigure}{.5\textwidth}
  \centering
  \includegraphics[scale=.55]{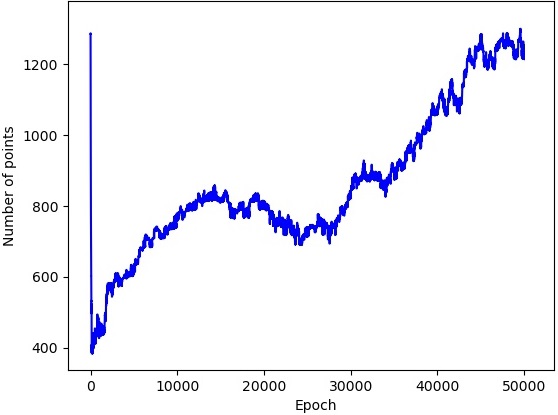}
  \caption{Sample set size using SS.}
  \label{fig:A4}
\end{subfigure}
\caption{Loss function versus epochs with the (a) classical sampling (CS) method and (b) stratified sampling (SS) method. (c) Mean squared error (MSE) for the exact solution ($u(x,t)=e^{-(x-t)^{2}}$) versus the epochs for the CS (dotted red line) and SS (solid blue line) methods. (d) Number of points selected versus epochs using the SS method. Both methods were used to approximate the solution of the advection equation (Eq. $(\ref{Acvec})$) over the domain $\Omega_{x}\times\Omega_{t}=[-200,800]\times [0,600]$ after 50,000 epochs.}

\label{fig:AD_Loss}
\end{figure}

As mentioned before, the advection equation meets the requirements of Corollary \ref{corpde}, which guarantees the improvement of all the weights and biases in the network during the optimization process, not only those in the first hidden layer. This result leads to a smaller number of epochs needed to obtain an acceptable solution for the advection equation compared to the number of epochs needed for the Fisher and Zeldovich equations (Eqs. $(\ref{Fisher})$ and $(\ref{Zeldovich})$). Additionally, Fig. \ref{fig:A2} illustrates smaller fluctuations in the changes of the loss function values in comparison to the Fisher and Zeldovich equations (Figs. \ref{fig:F2} and \ref{fig:Z6}).


\subsection{Fisher Equation}

Proposed in the context of population dynamics as a model for the spatial and temporal spread of an advantageous allele in an infinite one-dimensional medium \cite{fisher1937wave}, Fisher's equation is defined by
\begin{equation}
\label{FisherD}
v_{\tau}=Dv_{zz}+\rho v\left( 1-\frac{v}{K}\right)
\end{equation}
where $\tau$ is the time variable and $z\in\left(-\infty ,\infty\right)$ is the spatial variable, $D$ is the diffusion coefficient, $\rho$ represents a reactive rate coefficient, and $K$ is the carrying capacity. We consider solutions to Eq. $(\ref{FisherD})$ subject to the following initial and boundary conditions:
\begin{equation}
\label{FisherICBC}
\left\lbrace \begin{split}
& v_{\tau}=Dv_{zz}+\rho v\left( 1-\frac{v}{K}\right),\\
&\lim_{z\rightarrow -\infty} v(z,\tau) = 1,\\
&\lim_{z\rightarrow \infty} v(z,\tau) = 0, \\
& v(z,0)= v_{0}(z).
\end{split}\right.
\end{equation}
By rescaling the system with the following change of variables
\begin{equation*}
\begin{split}
x&=\sqrt{\frac{\rho}{D}}z, \ \ \ t=\rho \tau, \ \ \  u=Kv,
\end{split}
\end{equation*}
we rewrite Eq. $(\ref{FisherICBC})$ as
\begin{equation}
\label{Fisher}
\left\lbrace \begin{split}
& u_{t}=u_{xx}+u(1-u),\\
&\lim_{x\rightarrow -\infty} u(x,t) = 1,\\
&\lim_{x\rightarrow \infty} u(x,t) = 0, \\
& u(x,0)= u_{0}(x).
\end{split}\right.
\end{equation}
The new equation admits traveling wave solutions of the form $u(x,t) = w(x\pm ct) = w(x_{c})$,
where $w$ is decreasing and satisfies
\begin{equation*}
\begin{split}
\lim_{x_{c}\rightarrow -\infty} w(x_{c}) &= 1,\\
\lim_{x_{c}\rightarrow \infty} w(x_{c}) &= 0, \\
\end{split}
\end{equation*}
for dimensionless speeds $c\geq 2$.
The exact solution of Eq. $(\ref{Fisher})$
has been known 
\cite{olmos2016implicit, ablowitz1979explicit} and is given by
\begin{equation}
\label{solFish}
u(x,t)=\frac{1}{\left( 1+e^{\sqrt{\frac{1}{6}}x-\frac{5}{6}t}\right)^{2}}.
\end{equation}
Given the information of the domain and those values of $\rho$, $D$ and $K$ we approximate the solution of the original equation Eq. $(\ref{FisherICBC})$ by training a trial solution $u_T$ of Eq. $(\ref{Fisher})$ with the initial condition $u_{0}(x)=\left( 1+e^{\sqrt{\frac{1}{6}}x}\right)^{-2}$ using the above exact solution. 

Table \ref{tabFish} presents the results considering the same point density per area unit within the domain in each sample selected during training for the CS and SS algorithms. The table provides the results with various domain sizes considered.
\begin{center}
\begin{tabular}{c c c c}
\hline
Domain size & Epoch number &  Classical sampling & Stratified sampling\\
\hline
$\left[ -20,80\right]\times \left[ 0,30\right]$ & 20,000 & $7.83\times 10^{-7}$ & $3.17\times 10^{-6}$ \\
$\left[ -200,800\right]\times \left[ 0,300\right]$ & 40,000 & $6.09\times 10^{-7}$ & $1.13\times 10^{-6}$ \\
$\left[ -2000,8000\right]\times \left[ 0,3000\right]$ & 400,000 & $2.58\times 10^{-6}$ & $2.71\times 10^{-6}$ \\
\hline
\end{tabular}
\captionof{table}{Minimum loss function value for the Fisher equation (Eq. $(\ref{Fisher})$) with different domain sizes $\left[ x_{L},x_{R}\right]\times \left[ t_{0},t_{F}\right]$. The network for all trial solutions consists of three hidden layers with 40 neurons for each.}
\label{tabFish}
\end{center}

In comparison with the results in Table \ref{tabAD}, Table \ref{tabFish} reveals that CS sometimes has smaller loss function values than those of the SS method. However, as mentioned, lower loss function values do not imply better approximations of the solution to the PDE. As indicated in Fig. \ref{fig:Fish_loss}, depending on the selected sample, a sufficient approximation of the solution may present a higher or similar loss function value than a poor approximation. For example, both methods produce similar loss function values when considering the equation in large domains. However, when plotting both approximations and comparing them with the exact solution given by Eq. $(\ref{solFish})$, the SS method offers a better approximation of the analytical solution than the CS method. Moreover, on average, 400,000 epochs were required to obtain an acceptable approximation of the exact solution of Eq. $(\ref{Fisher})$ in the domain $\left[ -2000,8000\right]\times \left[ 0,3000\right]$ using the SS method. However, this is not the case when using CS techniques, as presented in Fig. \ref{fig:Fish_ts1}, which displays the trial solution using CS, where the maximum value reached by $U_{T}\left(x,t\right)$ is approximately 1.01 instead of 1, as it should be. On average, more than 700,000 epochs were required to obtain an acceptable approximation of Eq. $(\ref{Fisher})$ using CS over the same domain.\\

\begin{figure}[h]
\centering
\begin{subfigure}{.33\textwidth}
  \centering
  \includegraphics[scale=.55]{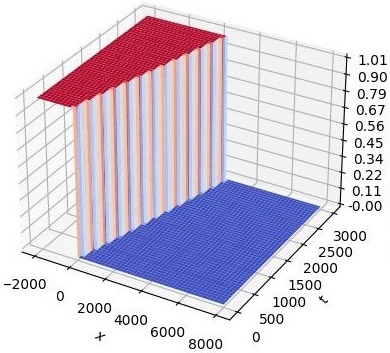}
  \caption{Trial solution using CS.}
  \label{fig:Fish_ts1}
\end{subfigure}%
\begin{subfigure}{.33\textwidth}
  \centering
  \includegraphics[scale=.55]{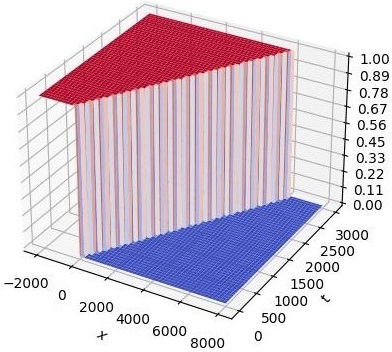}
  \caption{Trial solution using SS.}
  \label{fig:Fish_ts3}
\end{subfigure}%
\begin{subfigure}{.33\textwidth}
  \centering
  \includegraphics[scale=.55]{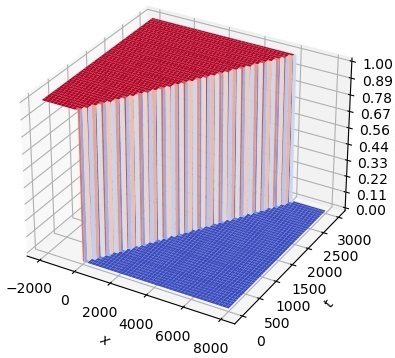}
  \caption{Exact solution.}
  \label{fig:Fish_ts2}
\end{subfigure}
\caption{Trial solutions with the (a) classical sampling (CS) and (b) stratified sampling (SS) algorithms after 400,000 epochs approximating (c) the exact solution to Fisher’s equation (Eq. $(\ref{Fisher})$) over the domain $\Omega_{x}\times\Omega_{t}=[-2000,8000]\times [0,3000]$ given by Eq. $(\ref{solFish})$.}
\label{fig:Fish_ts}
\end{figure}

Figure \ref{fig:Fish_IC} presents the benefits of performing optimization in two stages, where the first stage only approximates the initial condition. The initial condition and its approximation with both methods offered almost identical results, even when both methods resulted in two different traveling waves, as indicated in Fig. \ref{fig:Fish_ts}.\\
 
\begin{figure}[h]
\centering
\begin{subfigure}{.5\textwidth}
  \centering
  \includegraphics[scale=.56]{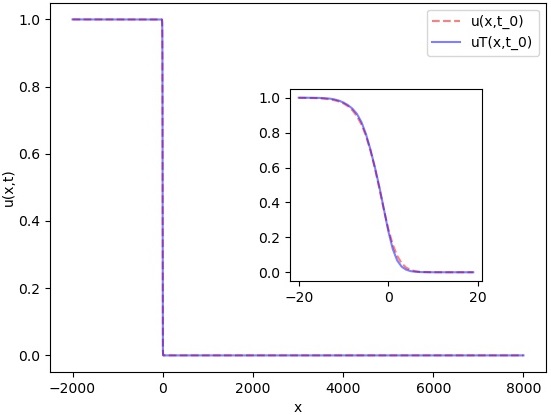}
  \caption{Initial condition using CS.}
  \label{fig:AD5}
\end{subfigure}%
\begin{subfigure}{.5\textwidth}
  \centering
  \includegraphics[scale=.56]{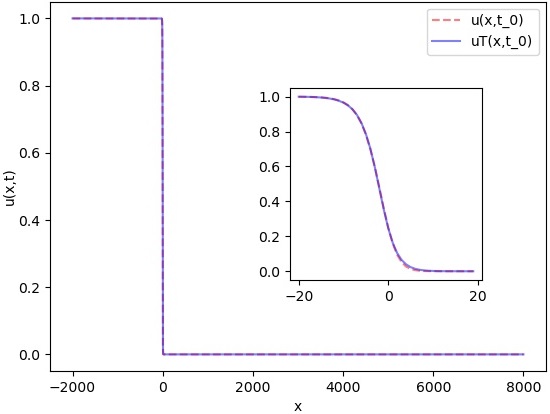}
  \caption{Initial condition using SS.}
  \label{fig:AD6}
\end{subfigure}
\caption{Graphs of the exact solution of the Fisher equation (Eq. $(\ref{Fisher})$) over the domain $\Omega_{x}\times\Omega_{t}=[-2000,8000]\times [0,3000]$ and its approximations using the (a) classical sampling (CS) method and (b) stratified sampling (SS) method after 400,000 epochs.}

\label{fig:Fish_IC}
\end{figure}

Figure \ref{fig:Fish_loss} depicts the dynamic that the loss function value presents during training using the CS (Fig. \ref{fig:F1}) and SS (Fig. \ref{fig:F2}) methods. The loss function seems saturated at almost any epoch of the optimization process when using CS techniques. Moreover, the loss function fluctuates more when using SS techniques, except for the period between 150,000 and 250,000 epochs during training. This outcome is partly due to the step size in which network parameters are updated using SS, which is considerably larger than the step size using the CS method. Figure \ref{fig:F3} displays the corresponding MSE with respect to Eq. $(\ref{solFish})$ for the CS method (dotted red line) and the SS method (solid blue line). In the figure, the proposed SS method yields better convergence behavior than the CS method. Figure \ref{fig:F4} illustrates how many sampling points are selected using SS as the training progresses, maintaining the same point density as the CS method, which used 50,000 points.

\begin{figure}[h]
\centering
\begin{subfigure}{.5\textwidth}
  \centering
  \includegraphics[scale=.55]{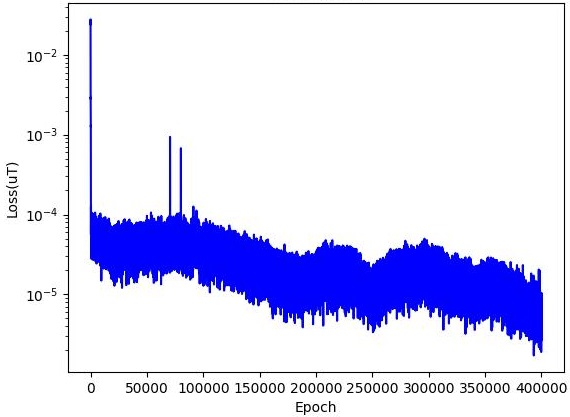}
  \caption{Loss function using CS.}
  \label{fig:F1}
\end{subfigure}%
\begin{subfigure}{.5\textwidth}
  \centering
  \includegraphics[scale=.55]{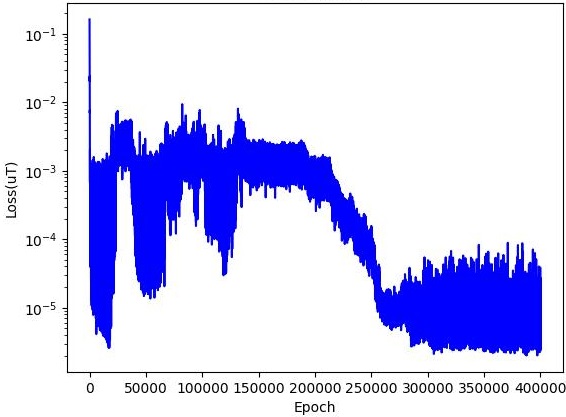}
  \caption{Loss function using SS.}
  \label{fig:F2}
\end{subfigure}
\begin{subfigure}{.5\textwidth}
  \centering
  \includegraphics[scale=.55]{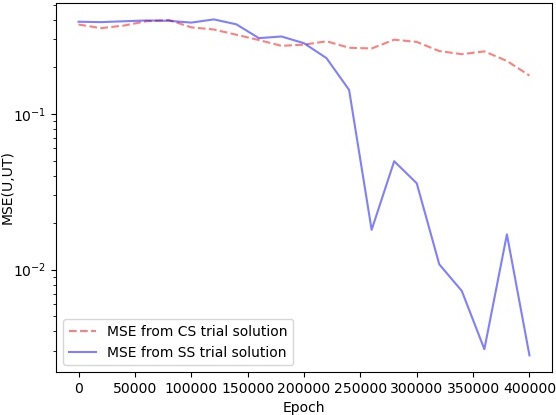}
  \caption{MSE respect to the exact solution.}
  \label{fig:F3}
\end{subfigure}%
\begin{subfigure}{.5\textwidth}
  \centering
  \includegraphics[scale=.55]{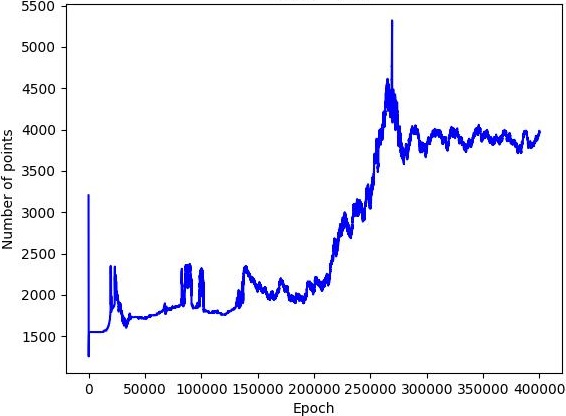}
  \caption{Sample set size using SS.}
  \label{fig:F4}
\end{subfigure}
\caption{Loss function versus epochs with the (a) classical sampling (CS) method and (b) stratified sampling (SS) method. (c) Mean squared error (MSE) for the exact solution (Eq. $(\ref{solFish})$) versus the epochs for the CS (dotted red line) and SS (solid blue line) methods. (d) Number of points selected versus epochs using the SS method. Both methods were used to approximate the solution of the Fisher equation (Eq. $(\ref{Fisher})$) over the domain $\Omega_{x}\times\Omega_{t}=[-2000,8000]\times [0,3000]$ after 400,000 epochs.}
\label{fig:Fish_loss}
\end{figure}

Unlike the advection equation, the Fisher equation does not meet the requirements of Corollary \ref{corpde}. Therefore, Theorem \ref{theopde} only guarantees an improvement in optimizing the parameters of the first hidden layer of $U_{T}(x,t)$. However, as indicated in Figs. \ref{fig:Fish_ts} and \ref{fig:Fish_loss}, this improvement is enough to optimize the trial solution in the sense of requiring a smaller number of epochs to obtain a sufficient approximation of the exact solution. Not being able to extend the result of Theorem \ref{theopde} to all the parameters in $U_{T}(x,t)$ may cause more unpredictable dynamics for the loss function value through network training, as depicted in Fig. \ref{fig:F2}.

\subsection{Zeldovich Equation}

The Zeldovich equation employed in the theory of combustion and detonation of gases as a method to study the propagation of a two-dimensional flame \cite{zeldovich1959theorv,gilding2004travelling} can be expressed in its simplest form by

\begin{equation}
\label{Zeldovich}
\left\lbrace \begin{split}
& u_{t}=u_{xx}+u^{2}\left( 1-u\right),\\
&\lim_{x\rightarrow -\infty} u(x,t) = 1,\\
&\lim_{x\rightarrow \infty} u(x,t) = 0, \\
& u(x,0)= u_{0}(x).
\end{split}\right.
\end{equation}
The unknown $u$ represents temperature, and the second term, $u^2(1-u)$, on the right-hand side of the differential equation corresponds to the heat generation by combustion. 

Equation $(\ref{Zeldovich})$ admits precisely one wavefront solution when considering a traveling speed greater than or equal to $\frac{1}{\sqrt{2}}$ and no such solution for any wave speed smaller than $\tfrac{1}{\sqrt{2}}$. More precisely, the exact solutions of Eq. $(\ref{Zeldovich})$ has been  explicitly found in multiple articles with, for example in \cite{kudryashov1993exact}, where considering the initial condition 
\begin{equation*}
u_{0}(x)=\frac{1}{1+e^{\frac{x}{\sqrt{2}}}}    
\end{equation*}
found that the solution of Eq. \ref{Zeldovich} is given by 
\begin{equation}
\label{Zelsol}
u(x,t)=\frac{1}{1+e^{\frac{x-vt}{\sqrt{2}}}},
\end{equation}
where $v$ represents the wavefront traveling speed, which in this particular example is $v=\tfrac{1}{\sqrt{2}}$.

Table \ref{tabZel} presents the results obtained with three hidden layers and 40 neurons per layer to approximate the solution to Eq. $(\ref{Zeldovich})$ using the uniform CS and SS algorithms. The table lists the minimum value of the loss function with different domain sizes.

\begin{center}
\begin{tabular}{c c c c}
\hline
Domain & Epoch & Classical sampling & Stratified sampling\\
\hline
$\left[ -25,175\right]\times \left[ 0,80\right]$ & 20,000 & $1.43\times 10^{-5}$ & $2.29\times 10^{-6}$ \\
$\left[ -250,1750\right]\times \left[ 0,800\right]$ & 60,000 & $9.16\times 10^{-7}$ & $5.81\times 10^{-6}$  \\
$\left[ -2500,17500\right]\times \left[ 0,8000\right]$ & 240,000 & $2.47\times 10^{-10}$ & $1.13\times 10^{-7}$ \\
\hline
\end{tabular}
\captionof{table}{Minimum loss function value for the Zeldovich equation (Eq. $(\ref{Zeldovich})$) with various domain sizes $\left[ x_{L},x_{R}\right]\times \left[ t_{0},t_{F}\right]$. The network for this equation consists of three hidden layers with 40 neurons each.}
\label{tabZel}
\end{center}

Similar to the observations during the loss function analysis for the Fisher equation, Table \ref{tabZel} lists smaller values of the loss function when using CS instead of SS. However, Fig. \ref{fig:Zel_Loss} reveals that, even when the CS techniques produce significantly smaller loss function values than the SS techniques when plotting both approximations and comparing them with the exact solution given by Eq. $(\ref{Zelsol})$, SS approximates the analytical solution better than CS. On average, $200,000$ epochs were required to obtain an acceptable approximation of the exact solution of Eq. $(\ref{Zeldovich})$ in the domain $\left[ -2500,17500\right]\times \left[ 0,8000\right]$ using SS. However, this is not the case using CS techniques, as indicated in Fig. \ref{fig:Zel}, which presents the exact and trial solutions obtained using both methods. When using CS methods over the same domain, on average, more than $600,000$ epochs were needed to obtain an acceptable approximation of the solution to Eq. $(\ref{Zeldovich})$ given by Eq. $(\ref{Zelsol})$.

Figure \ref{fig:Zel} illustrates the trial solutions obtained with the uniform CS method (Fig. \ref{fig:Z1}) and the SS method (Fig. \ref{fig:Z2}). Figure \ref{fig:Z3} depicts the exact solution. For the results, the domain size is $\Omega_{x}\times\Omega_{t}=[-2500,17500]\times [0,8000]$ and the solutions were obtained in $240,000$ epochs. The trial solution with the proposed SS method yields an accurate traveling wave solution that is close to the exact solution. The traveling wave speed with the uniform CS method is greater than that of the exact solution.\\

\begin{figure}[h]
\centering
\begin{subfigure}{.33\textwidth}
  \centering
  \includegraphics[scale=.56]{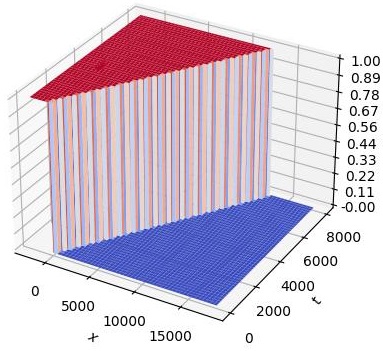}
  \caption{Trial solution using CS.}
  \label{fig:Z1}
\end{subfigure}%
\begin{subfigure}{.33\textwidth}
  \centering
  \includegraphics[scale=.56]{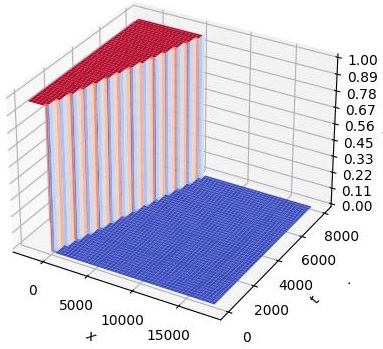}
  \caption{Trial solution using SS.}
  \label{fig:Z2}%
\end{subfigure}
\begin{subfigure}{.33\textwidth}
  \centering
  \includegraphics[scale=.56]{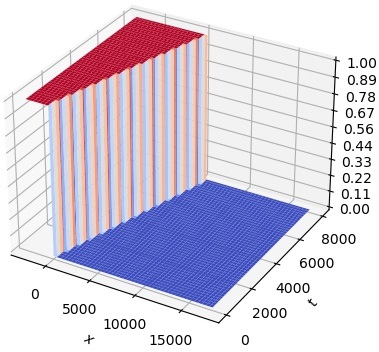}
  \caption{Exact solution.}
  \label{fig:Z3}
\end{subfigure}
\caption{Trial solutions obtained using the (a) classical uniform sampling method and (b) stratified sampling method after 240,000 epochs. (c) Exact solution
to the Zeldovich equation $(\ref{Zeldovich})$ 
over the domain $\Omega_{x}\times\Omega_{t}=[-2500,17500]\times [0,8000]$, given by Eq. $(\ref{Zelsol})$.}
\label{fig:Zel}
\end{figure}

Figure \ref{fig:Zel_IC} presents the initial condition and its approximation with the CS (Fig. \ref{fig:Zi5}) and SS (Fig. \ref{fig:Zi6}) methods, including a close view of these approximations of the initial condition. The dotted red line represents the exact initial condition, and the solid blue line approximates it. Both methods provide a good approximation for most of the domain of the PDE, only separating from the exact solution by a nonnegligible distance within the interval $(-20, 20)$. However, the proposed SS method yields a sufficient approximation of the initial condition compared to CS.\\

\begin{figure}[h]
\centering
\begin{subfigure}{.5\textwidth}
  \centering
  \includegraphics[scale=.56]{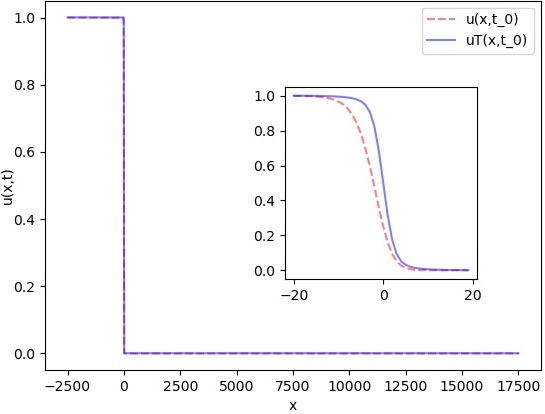}
  \caption{Initial condition using CS.}
  \label{fig:Zi5}
\end{subfigure}%
\begin{subfigure}{.5\textwidth}
  \centering
  \includegraphics[scale=.56]{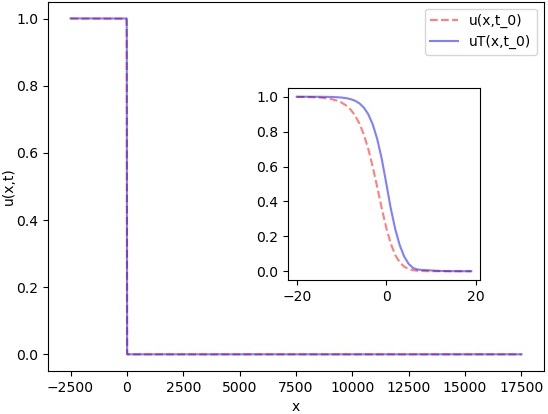}
  \caption{Initial condition using SS.}
  \label{fig:Zi6}
\end{subfigure}
\caption{Approximations of the initial condition of the Zeldovich equation (Eq. $(\ref{Zeldovich})$) over the domain $\Omega_{x}\times\Omega_{t}=[-2500,17500]\times [0,8000]$ obtained using the (a) classic sampling method and (b) stratified sampling method after 240,000 epochs. }
\label{fig:Zel_IC}
\end{figure}

Figure \ref{fig:Zel_Loss} illustrates the dynamic that the loss function value presents during the training progress using the CS (Fig. \ref{fig:Z5}) and SS (Fig. \ref{fig:Z6}) methods. The loss function fluctuates more with the CS techniques than the SS technique. Figure \ref{fig:Z7} depicts the corresponding MSE with respect to Eq. $(\ref{Zelsol})$ for the CS (dotted red line) and SS (solid blue line) methods. In the figure, the proposed SS method yields better convergence behavior than the CS method. Figure \ref{fig:z8} indicates how many sampling points were selected using SS as the training progressed, maintaining the same point density as the CS method, which used $80,000$ points. Figure \ref{fig:Zel_Loss} shows similar results to Fig. \ref{fig:Fish_loss}, this is because, like the Fisher equation, the Zeldovich equation does not meet the requirements of Corollary \ref{corpde}.

\begin{figure}[h]
\centering
\begin{subfigure}{.5\textwidth}
  \centering
  \includegraphics[scale=.55]{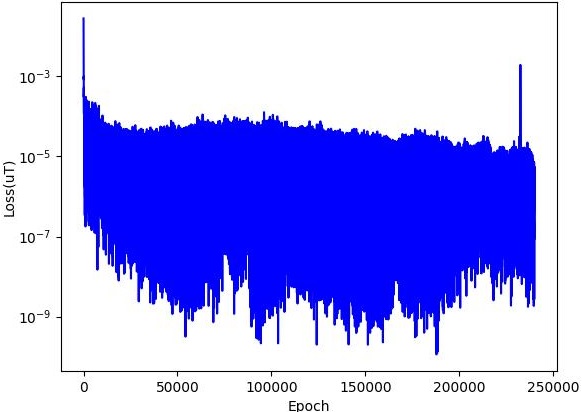}
  \caption{Initial condition using CS.}
  \label{fig:Z5}
\end{subfigure}%
\begin{subfigure}{.5\textwidth}
  \centering
  \includegraphics[scale=.55]{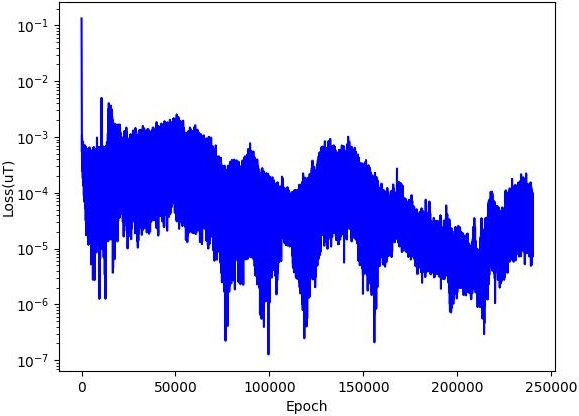}
  \caption{Initial condition using SS.}
  \label{fig:Z6}
\end{subfigure}
\begin{subfigure}{.5\textwidth}
  \centering
  \includegraphics[scale=.55]{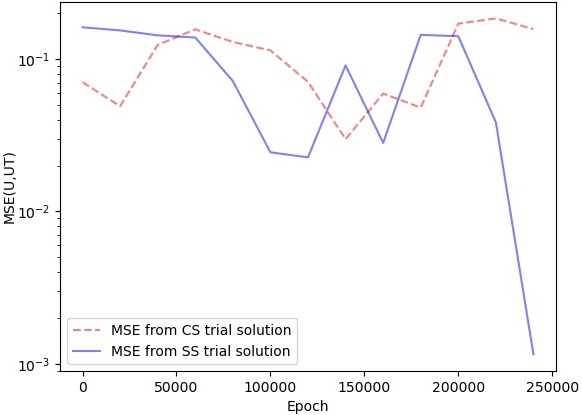}
  \caption{MSE respect to the exact solution.}
  \label{fig:Z7}
\end{subfigure}%
\begin{subfigure}{.5\textwidth}
  \centering
  \includegraphics[scale=.55]{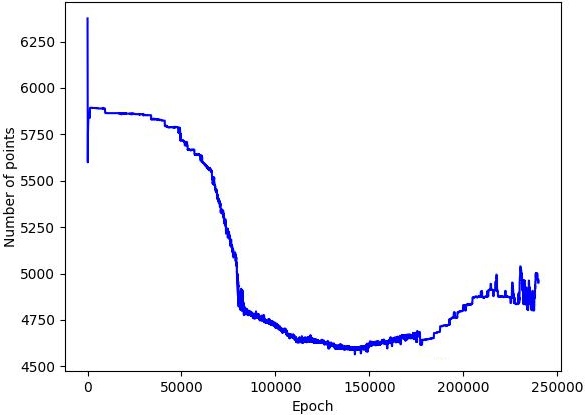}
  \caption{Sample set size using SS.}
  \label{fig:z8}
\end{subfigure}
\caption{Loss function versus epochs with the (a) classical sampling (CS) method and (b) stratified sampling (SS) method. (c) Mean squared error (MSE) for the exact solution (Eq. $(\ref{Zelsol})$) versus the epochs for the CS (dotted red line) and SS (solid blue line) methods. (d) Number of points selected versus epochs using the SS method. Both methods were used to approximate the solution of the Zeldovich equation (Eq. $(\ref{Zeldovich})$) over the domain $\Omega_{x}\times\Omega_{t}=[-2500,17500]\times [0,8000]$ after 240,000 epochs.}
\label{fig:Zel_Loss}
\end{figure}


\section{Conclusions}

This paper proposes a stratified sampling algorithm to enhance the training process of machine learning models in solving two-scale PDEs. This algorithm divides the domain of the PDE into two scales based on the notion of active and diminishing gradient zones of the trial solution that consider only the weights and biases in the first hidden layer of the neural network. The stratified sampling algorithm was used, splitting training into two stages. During the first stage, the algorithm was implemented to create a sample within the domain for the initial condition, and in the second stage, a sample was created considering the entire domain of the PDE. This approach was taken to employ a nonrandom initial configuration of the weights and biases in the network when starting the training for the trial solution in the entire domain of the PDE. The initial configuration provides a good approximation of the initial condition and a small loss function value. The numerical examples in Section 5 provide better approximations of the solutions to different PDEs when using the proposed stratified sampling algorithm than those obtained using the classical sampling method.

The idea behind the proposed technique is based on Theorem \ref{theopde} and Corollary \ref{corpde}, guaranteeing a decrease in the gradient of the loss function, considering the points within the diminishing gradient zone of a trial solution for the gradient obtained using a sample within the active gradient zone of the same trial solution. The proof of these results relies on the properties of the trial solution within the regions considered the active and diminishing gradient zones that consider the weights and biases of the first hidden layer only, obtained from Theorem \ref{Coric} and Corollary \ref{Cortheo1}.

The results of this work were obtained considering that the solution domain is part of $\mathbb{R}^{2}$ but can be extended to solutions whose domains are within $\mathbb{R}^{n}$. The results were achieved by considering a neural network whose input is in $\mathbb{R}^{n}$ (e.g., $\left( x_{1},\ldots,x_{n-1},t\right)$) as in the general case in Eq. (\ref{NNs}) and performing the same analysis for each $x_{i}$ as that performed for $x\in\mathbb{R}$ throughout this paper.


\section{Acknowledments}

\noindent The authors would like to acknowledge ACARUS from the Universidad de Sonora, for providing their support and facilities during the numerical computations. This work was supported by NRF of Korea under the grant number 2021R1A2C3009648 and NRF grant number NRF2021R1A6A1A1004294412. JHJ was also supported partially by NRF grant by the Korea government (MSIT) (RS-2023-00219980).



\appendix


\section{Pseudocodes}
\label{AppendixA}

Algorithm \ref{Algorithm:IC} shows the pseudo-code used to create the sample used during the optimization process for the verification of the initial condition. 


\begin{algorithm}
\caption{Sample selection algorithm used to train the initial condition.}
\label{Algorithm:IC}
\DontPrintSemicolon
\KwIn{Trial solution $U_{T} {(} x,t {)}$ and desire average density of points per square unit $d_{p}$}
\KwOut{Sample set $S_{IC}\in\left[ x_{L},x_{R}\right]$ generated over the active gradient zone of $U_{0} {(} x {)}$, with an average density of points per square unit of approximately $d_{p}$.}
\SetKwBlock{Begin}{function}{end function}
\Begin($\textrm{S\_{IC}} {(} U_{T} {(} x,t {)} ,d_{p}, \epsilon{)}$){
    \ \ \For{$j\in \left\lbrace 1,\ldots ,h_{1} \right\rbrace$}{
        \ \ $C^{(j)} \leftarrow -\dfrac{\left( t_{0}\omega_{1}^{(j,2)}+b_{1}^{(j)}\right)}{\omega_{1}^{(j,1)}}$\\
        $\delta^{(j)} \leftarrow \dfrac{\epsilon}{\omega_{1}^{(j,1)}}$\\
        $\left\langle x_{L}^{(j)},x_{R}^{(j)}\right\rangle \leftarrow \left\langle\max{\left\lbrace x_{L}, C^{(j)}-\delta^{(j)}\right\rbrace},\min{\left\lbrace x_{R}, C^{(j)}+\delta^{(j)}\right\rbrace}\right\rangle$
    }
    \ \ $\left\lbrace\left\langle x_{L}^{\left( j^{\prime}\right)},x_{R}^{\left( j^{\prime}\right)}\right\rangle\right\rbrace_{j^{\prime}=1}^{h_{1}} \leftarrow \mathrm{sort}_{x_{L}^{(j)}}\left( \left\lbrace\left\langle x_{L}^{(j)},x_{R}^{(j)}\right\rangle\right\rbrace_{j=1}^{h_{1}}\right)$\\
    $j^{\prime} \leftarrow 1$\\
    \While{$j^{\prime}\leq h_{1}$}{  
        \ \ $k \leftarrow 0$\\
        $h_{1}^{\prime} \leftarrow 0$\\
        \While{$k \leq L_{1}-j$ \textbf{{\normalfont\textbf{and}}} $x_{R}^{\left( j^{\prime}\right)} \geq x_{L}^{\left( j^{\prime}+k+1\right)}$}{
            \ \ $k \leftarrow k+1$
        }
        \ \ $n_{j^{\prime}} \leftarrow \left\lceil d_{p}\left( \max{\left\lbrace x_{R,i}^{\left( j^{\prime}\right)},x_{R,i}^{\left( j^{\prime}+k\right)}\right\rbrace} - x_{L,i}^{\left( j^{\prime}\right)}\right)\right\rceil$\\
        $S^{\left( h_{1}^{\prime}+1\right)} \leftarrow \left\lbrace (x_{l},t_{0})\right\rbrace_{l=1}^{n_{j^{\prime}}} \sim U\left(\left( x_{L}^{\left( j^{\prime}\right)},\max{\left\lbrace x_{R}^{\left( j^{\prime}\right)},x_{R}^{\left( j^{\prime}+k\right)}\right\rbrace}\right)\times\lbrace t_{0}\rbrace\right)$\\
        $j^{\prime} \leftarrow j^{\prime}+k+1$\\
        $h_{1}^{\prime} \leftarrow h_{1}^{\prime}+1$
    }
    \ \ $S_{IC} \leftarrow \bigcup_{j=1}^{h_{1}^{\prime}} S^{\left( j^{\prime}\right)}$\\
    \Return{$S_{IC}$}
}
\end{algorithm}

\newpage

Algorithm \ref{Algorithm:PDE} shows the pseudo-code used to create the sample used during the optimization process for the verification of the differential equation.

\begin{algorithm}
\caption{Sample selection algorithm used to train the PDE.}
\label{Algorithm:PDE}
\DontPrintSemicolon
\KwIn{Trial solution $U_{T} {(} x,t {)}$ and desire average density of points per square unit $d_{p}$}
\KwOut{Sample set $S_{PDE}\in\left( x_{L},x_{R}\right)\times\left( t_{0},t_{F}\right)$ generated over the active gradient zone of $U_{T} {(} x,t {)}$, with an average density of points per square unit of approximately $d_{p}$.}
\SetKwBlock{Begin}{function}{end function}
\Begin($\textrm{S\_PDE} {(} U_{T} {(} x,t {)} ,d_{p},\epsilon, N {)}$){
    \ \ $\Delta_{t} \leftarrow \dfrac{(t_{F}-t_{0})}{N}$\\
    \For{$j\in \left\lbrace1,\ldots ,h_{1}\right\rbrace$}{
        \ \ $\delta^{(j)} \leftarrow \dfrac{\epsilon}{\omega_{1}^{(j,1)}}$
    }
    \ \ \For{$i\in\left\lbrace 0,\ldots ,N-1\right\rbrace$}{
        \ \ $t_{i} \leftarrow  t_{0} + i\Delta_{t}$\\
        $t_{i+1} \leftarrow  t_{i} + \Delta_{t}$\\
        \For{$j\in \left\lbrace1,\ldots ,h_{1}\right\rbrace$}{
            \ \ $C_{i}^{(j)} \leftarrow -\dfrac{\left( t_{i}\omega_{1}^{(j,2)}+b_{1}^{(j)}\right)}{\omega_{1}^{(j,1)}}$\\
            $C_{i+1}^{(j)} \leftarrow -\dfrac{\left( t_{i+1}\omega_{1}^{(j,2)}+b_{1}^{(j)}\right)}{\omega_{1}^{(j,1)}}$\\
            $x_{L,i}^{(j)} \leftarrow \min{ \left\lbrace C_{i}^{(j)}-\delta^{(j)}, C_{i+1}^{(j)}-\delta^{(j)} \right\rbrace}$\\
            $x_{R,i}^{(j)} \leftarrow \max{ \left\lbrace C_{i}^{(j)}+\delta^{(j)}, C_{i+1}^{(j)}+\delta^{(j)} \right\rbrace}$\\
            $\left\langle x_{L,i}^{(j)},x_{R,i}^{(j)}\right\rangle \leftarrow \left\langle\max{\left\lbrace x_{L}, x_{L,i}^{(j)}\right\rbrace},\min{\left\lbrace x_{R}, x_{R,i}^{(j)}\right\rbrace}\right\rangle$
        }
        \ \ $\left\lbrace\left\langle x_{L,i}^{\left( j^{\prime}\right)},x_{R,i}^{\left( j^{\prime}\right)}\right\rangle\right\rbrace_{j^{\prime}=1}^{h_{1}} \leftarrow \mathrm{sort}_{x_{L,i}^{(j)}}\left( \left\lbrace\left\langle x_{L,i}^{(j)},x_{R,i}^{(j)}\right\rangle\right\rbrace_{j=1}^{h_{1}}\right)$\\
        $j^{\prime} \leftarrow 1$\\
        \While{$j^{\prime}\leq h_{1}$}{  
            \ \ $k \leftarrow 0$\\
            $h_{1,i}^{\prime} \leftarrow 0$\\
            \While{$k \leq L_{1}-j$ \textbf{{\normalfont\textbf{and}}} $x_{R,i}^{\left( j^{\prime}\right)} \geq x_{L,i}^{\left( j^{\prime}+k+1\right)}$}{
                \ \ $k \leftarrow k+1$
            }
            \ \ $n_{j^{\prime}} \leftarrow \left\lceil d_{p}\left( \max{\left\lbrace x_{R,i}^{\left( j^{\prime}\right)},x_{R,i}^{\left( j^{\prime}+k\right)}\right\rbrace} - x_{L,i}^{\left( j^{\prime}\right)}\right)\Delta_{t}\right\rceil$\\
            $S_{i}^{\left( h_{1,i}^{\prime}+1\right)} \leftarrow \left\lbrace \left( x_{l,i}, t_{l,i}\right)\right\rbrace_{l=1}^{n_{j^{\prime}}} \sim U\left(\left( x_{L,i}^{\left( j^{\prime}\right)},\max{\left\lbrace x_{R,i}^{\left( j^{\prime}\right)},x_{R,i}^{\left( j^{\prime}+k\right)}\right\rbrace}\right)\times\left( t_{i}, t_{i+1} \right]\right)$\\
            $j^{\prime} \leftarrow j^{\prime}+k+1$\\
            $h_{1,i}^{\prime} \leftarrow h_{1,i}^{\prime}+1$
        }
    }
    \ \ $S_{PDE} \leftarrow \bigcup_{i=1}^{N}\bigcup_{j=1}^{h_{1,i}^{\prime}} S_{i}^{\left( j\right)}$\\
    \Return{$S_{PDE}$}
}
\end{algorithm}





\bibliographystyle{elsarticle-num}
\bibliography{cas-refs}

\begin{thebibliography}{10}
\expandafter\ifx\csname url\endcsname\relax
  \def\url#1{\texttt{#1}}\fi
\expandafter\ifx\csname urlprefix\endcsname\relax\def\urlprefix{URL }\fi
\expandafter\ifx\csname href\endcsname\relax
  \def\href#1#2{#2} \def\path#1{#1}\fi

\bibitem{majda2008applied}
A.~J. Majda, C.~Franzke, B.~Khouider, An applied mathematics perspective on stochastic modelling for climate, Philosophical Transactions of the Royal Society A: Mathematical, Physical and Engineering Sciences 366~(1875) (2008) 2427--2453.

\bibitem{fish2010multiscale}
J.~Fish, Multiscale methods: bridging the scales in science and engineering, Oxford University Press, 2010.

\bibitem{yu2002multiscale}
Q.~Yu, J.~Fish, Multiscale asymptotic homogenization for multiphysics problems with multiple spatial and temporal scales: a coupled thermo-viscoelastic example problem, International journal of solids and structures 39~(26) (2002) 6429--6452.

\bibitem{hou2003numerical}
T.~Y. Hou, Numerical approximations to multiscale solutions in partial differential equations, Frontiers in Numerical Analysis: Durham 2002 (2003) 241--301.

\bibitem{kevorkian2012multiple}
J.~K. Kevorkian, J.~D. Cole, Multiple scale and singular perturbation methods, Vol. 114, Springer Science \& Business Media, 2012.

\bibitem{vadyala2022physics}
S.~R. Vadyala, S.~N. Betgeri, N.~P. Betgeri, Physics-informed neural network method for solving one-dimensional advection equation using pytorch, Array 13 (2022) 100110.

\bibitem{glorot2010understanding}
X.~Glorot, Y.~Bengio, Understanding the difficulty of training deep feedforward neural networks, in: Proceedings of the thirteenth international conference on artificial intelligence and statistics, JMLR Workshop and Conference Proceedings, 2010, pp. 249--256.

\bibitem{pascanu2013difficulty}
R.~Pascanu, T.~Mikolov, Y.~Bengio, On the difficulty of training recurrent neural networks, in: International conference on machine learning, Pmlr, 2013, pp. 1310--1318.

\bibitem{fisher1937wave}
R.~A. Fisher, The wave of advance of advantageous genes, Annals of eugenics 7~(4) (1937) 355--369.

\bibitem{zeldovich1959theorv}
Y.~B. Zeldovich, G.~Barenblatt, Theorv of flame propagation, Combustion and flame 3 (1959) 61--74.

\bibitem{gilding2004travelling}
B.~H. Gilding, R.~Kersner, Travelling waves in nonlinear diffusion-convection reaction, Vol.~60, Springer Science \& Business Media, 2004.

\bibitem{le2015simple}
Q.~V. Le, N.~Jaitly, G.~E. Hinton, A simple way to initialize recurrent networks of rectified linear units, arXiv preprint arXiv:1504.00941 (2015).

\bibitem{van2022optimally}
R.~van~der Meer, C.~W. Oosterlee, A.~Borovykh, Optimally weighted loss functions for solving pdes with neural networks, Journal of Computational and Applied Mathematics 405 (2022) 113887.

\bibitem{ioffe2015batch}
S.~Ioffe, C.~Szegedy, Batch normalization: Accelerating deep network training by reducing internal covariate shift, in: International conference on machine learning, pmlr, 2015, pp. 448--456.

\bibitem{hochreiter1998vanishing}
S.~Hochreiter, The vanishing gradient problem during learning recurrent neural nets and problem solutions, International Journal of Uncertainty, Fuzziness and Knowledge-Based Systems 6~(02) (1998) 107--116.

\bibitem{plaut1986experiments}
D.~C. Plaut, et~al., Experiments on learning by back propagation. (1986).

\bibitem{wang2020comprehensive}
Q.~Wang, Y.~Ma, K.~Zhao, Y.~Tian, A comprehensive survey of loss functions in machine learning, Annals of Data Science (2020) 1--26.

\bibitem{liberty2016stratified}
E.~Liberty, K.~Lang, K.~Shmakov, Stratified sampling meets machine learning, in: International conference on machine learning, PMLR, 2016, pp. 2320--2329.

\bibitem{may2010data}
R.~J. May, H.~R. Maier, G.~C. Dandy, Data splitting for artificial neural networks using som-based stratified sampling, Neural Networks 23~(2) (2010) 283--294.

\bibitem{xu2018splitting}
Y.~Xu, R.~Goodacre, On splitting training and validation set: a comparative study of cross-validation, bootstrap and systematic sampling for estimating the generalization performance of supervised learning, Journal of analysis and testing 2~(3) (2018) 249--262.

\bibitem{fraboni2021clustered}
Y.~Fraboni, R.~Vidal, L.~Kameni, M.~Lorenzi, Clustered sampling: Low-variance and improved representativity for clients selection in federated learning, in: International Conference on Machine Learning, PMLR, 2021, pp. 3407--3416.

\bibitem{hanna2022residual}
J.~M. Hanna, J.~V. Aguado, S.~Comas-Cardona, R.~Askri, D.~Borzacchiello, Residual-based adaptivity for two-phase flow simulation in porous media using physics-informed neural networks, Computer Methods in Applied Mechanics and Engineering 396 (2022) 115100.

\bibitem{zeng2022adaptive}
S.~Zeng, Z.~Zhang, Q.~Zou, Adaptive deep neural networks methods for high-dimensional partial differential equations, Journal of Computational Physics 463 (2022) 111232.

\bibitem{wu2023comprehensive}
C.~Wu, M.~Zhu, Q.~Tan, Y.~Kartha, L.~Lu, A comprehensive study of non-adaptive and residual-based adaptive sampling for physics-informed neural networks, Computer Methods in Applied Mechanics and Engineering 403 (2023) 115671.

\bibitem{mao2023physics}
Z.~Mao, X.~Meng, Physics-informed neural networks with residual/gradient-based adaptive sampling methods for solving partial differential equations with sharp solutions, Applied Mathematics and Mechanics 44~(7) (2023) 1069--1084.

\bibitem{raissi2019physics}
M.~Raissi, P.~Perdikaris, G.~E. Karniadakis, Physics-informed neural networks: A deep learning framework for solving forward and inverse problems involving nonlinear partial differential equations, Journal of Computational physics 378 (2019) 686--707.

\bibitem{cho2021traveling}
S.~W. Cho, H.~J. Hwang, H.~Son, Traveling wave solutions of partial differential equations via neural networks, Journal of Scientific Computing 89~(1) (2021) 21.

\bibitem{deshpande2018artificial}
A.~Deshpande, M.~Kumar, Artificial intelligence for big data: Complete guide to automating big data solutions using artificial intelligence techniques, Packt Publishing Ltd, 2018.

\bibitem{cybenko1989approximation}
G.~Cybenko, Approximation by superpositions of a sigmoidal function, Mathematics of control, signals and systems 2~(4) (1989) 303--314.

\bibitem{hornik1989multilayer}
K.~Hornik, M.~Stinchcombe, H.~White, Multilayer feedforward networks are universal approximators, Neural networks 2~(5) (1989) 359--366.

\bibitem{ackermann2023deep}
J.~Ackermann, A.~Jentzen, T.~Kruse, B.~Kuckuck, J.~L. Padgett, Deep neural networks with relu, leaky relu, and softplus activation provably overcome the curse of dimensionality for kolmogorov partial differential equations with lipschitz nonlinearities in the {$L^{p}$}-sense, arXiv preprint arXiv:2309.13722 (2023).

\bibitem{lei2022solving}
Z.~Lei, L.~Shi, C.~Zeng, Solving parametric partial differential equations with deep rectified quadratic unit neural networks, Journal of Scientific Computing 93~(3) (2022) 80.

\bibitem{lu2019deeponet}
L.~Lu, P.~Jin, G.~E. Karniadakis, Deeponet: Learning nonlinear operators for identifying differential equations based on the universal approximation theorem of operators, arXiv preprint arXiv:1910.03193 (2019).

\bibitem{montavon2019layer}
G.~Montavon, A.~Binder, S.~Lapuschkin, W.~Samek, K.-R. M{\"u}ller, Layer-wise relevance propagation: an overview, Explainable AI: interpreting, explaining and visualizing deep learning (2019) 193--209.

\bibitem{hanin2018neural}
B.~Hanin, Which neural net architectures give rise to exploding and vanishing gradients?, Advances in neural information processing systems 31 (2018).

\bibitem{kong2017hexpo}
S.~Kong, M.~Takatsuka, Hexpo: A vanishing-proof activation function, in: 2017 International Joint Conference on Neural Networks (IJCNN), IEEE, 2017, pp. 2562--2567.

\bibitem{squartini2003preprocessing}
S.~Squartini, A.~Hussain, F.~Piazza, Preprocessing based solution for the vanishing gradient problem in recurrent neural networks, in: Proceedings of the 2003 International Symposium on Circuits and Systems, 2003. ISCAS'03., Vol.~5, IEEE, 2003, pp. V--V.

\bibitem{mo2021quantifying}
F.~Mo, A.~Borovykh, M.~Malekzadeh, S.~Demetriou, D.~G{\"u}nd{\"u}z, H.~Haddadi, Quantifying and localizing usable information leakage from neural network gradients, arXiv preprint arXiv:2105.13929 (2021).

\bibitem{rennie1969stirling}
B.~C. Rennie, A.~J. Dobson, On stirling numbers of the second kind, Journal of Combinatorial Theory 7~(2) (1969) 116--121.

\bibitem{broder1984r}
A.~Z. Broder, The r-stirling numbers, Discrete Mathematics 49~(3) (1984) 241--259.

\bibitem{minai1993derivatives}
A.~A. Minai, R.~D. Williams, On the derivatives of the sigmoid, Neural Networks 6~(6) (1993) 845--853.

\bibitem{faa1855sullo}
F.~Fa{\`a}~di Bruno, Sullo sviluppo delle funzioni, Annali di scienze matematiche e fisiche 6~(1) (1855) 479--480.

\bibitem{di1857note}
F.~Fa{\`a}~di Bruno, Note sur une nouvelle formule de calcul diff{\'e}rentiel, Quarterly J. Pure Appl. Math 1~(359-360) (1857) 12.

\bibitem{constantine1996multivariate}
G.~Constantine, T.~Savits, A multivariate faa di bruno formula with applications, Transactions of the American Mathematical Society 348~(2) (1996) 503--520.

\bibitem{olmos2016implicit}
D.~Olmos-Liceaga, I.~Segundo-Caballero, An implicit pseudospectral scheme to solve propagating fronts in reaction-diffusion equations, Numerical Methods for Partial Differential Equations 32~(1) (2016) 86--105.

\bibitem{ablowitz1979explicit}
M.~J. Ablowitz, A.~Zeppetella, Explicit solutions of fisher's equation for a special wave speed, Bulletin of Mathematical Biology 41~(6) (1979) 835--840.

\bibitem{kudryashov1993exact}
N.~A. Kudryashov, Exact solutions of a family of fisher equations, Theoretical and Mathematical Physics 94~(2) (1993) 211--218.

\end{thebibliography}






\end{document}